\documentclass[a4paper,10pt,reqno]{amsart}
\usepackage[english,activeacute]{babel}
\usepackage[utf8]{inputenc}
\usepackage[foot]{amsaddr}
\usepackage{hyperref}

\usepackage{enumitem}
\setlist[itemize]{topsep=0.2em, itemsep=0.2em, leftmargin=2em}
\setlist[enumerate]{topsep=0.2em, itemsep=0.2em, leftmargin=2em}
\setlist[description]{topsep=0.2em, itemsep=0.2em, leftmargin=2em}

\usepackage{latexsym}

\usepackage{graphicx}
\usepackage{color}
\usepackage{tikz-cd}
\usetikzlibrary{matrix}

\usepackage[a4paper,top=3.3cm,bottom=3cm,left=1.8cm,right=1.8cm,bindingoffset=5mm]{geometry}
\usepackage{marginnote}
\definecolor{alert}{rgb}{0.8,0,0.3}
\newcommand{\alert}[1]{%
	\marginpar{%
		\ifodd\value{page} \raggedright \else \raggedleft \fi
		\footnotesize{\textcolor{alert}{#1}}
	}
}

\newcommand{\N}{\mathbb{N}}

\newcommand{\R}{\mathbb{R}}

\renewcommand{\H}{\mathbb{H}}
\newcommand{\df}{\mathrm{d}}

\newcommand{\E}{\mathbb{E}}

\newcommand{\X}{\mathfrak{X}}

\newcommand{\Div}{\mathop{\rm div}\nolimits}

\newcommand{\Length}{\mathop{\rm Length}\nolimits}

\newcommand{\Vol}{\mathop{\rm Vol}\nolimits}

\newcommand{\Nil}{\mathrm{Nil}_3}

\newcommand{\PSL}{\widetilde{\mathrm{SL}}_2(\mathbb{R})}
\newcommand{\Sol}{\mathrm{Sol}_3}

\newcommand{\Norm}[1]{ \left\lVert {#1} \right\rVert}
\newcommand{\Prod}[1]{ \left\langle {#1} \right\rangle}
\newcommand{\prodesc}[2]{\left\langle {#1},{#2} \right\rangle}

\newtheorem{theorem}{Theorem}[section]

\newtheorem{corollary}[theorem]{Corollary}
\newtheorem{lemma}[theorem]{Lemma}

\theoremstyle{definition}
\newtheorem{definition}[theorem]{Definition}

\theoremstyle{remark}
\newtheorem{remark}[theorem]{Remark}
\theoremstyle{proof}
\newtheorem{example}[theorem]{Example}

\numberwithin{equation}{section}
\title{Minimal graphs over non-compact domains in 3-manifolds with a Killing vector field}
\date{}

\author{Andrea Del~Prete}
\address{Dipartimento di Matematica "Felice Casorati"\\ Universit\`{a} degli Studi di Pavia\\ Via Adolfo Ferrata, 5 -- 27100 Pavia Italy}
\email{andrea.delprete@unipv.it}



\begin{document}
	
	\title[Minimal graphs over non-compact domains in 3-manifolds with a Killing vector field]{Minimal graphs over non-compact domains in 3-manifolds with a Killing vector field}
	
	\begin{abstract}
			Let $\mathbb{E}$ be a connected and orientable Riemannian 3-manifold with a non-singular Killing vector field whose associated one-parameter group of the isometries of $\mathbb{E}$ acts freely and properly on $\E$. Then, there exists a Killing Submersion from $\E$ onto a connected and orientable surface $M$ whose fibers are the integral curves of the Killing vector field. 
			In this setting, assuming that $M$ is non-compact and the fibers have infinite length, we solve the Dirichlet problem for minimal Killing graphs over certain unbounded domains of $M$, prescribing piecewise continuous boundary values.
			We obtain general Collin-Krust type estimates.
			In the particular case of the Heisenberg group, we prove a uniqueness result for minimal Killing graphs with bounded boundary values over a strip.
			We also prove that isolated singularities of Killing graphs with prescribed mean curvature are removable.
	\end{abstract}
	
	
	
	\maketitle
	\textbf{Acknowledgments} This work is part of the author's PhD thesis and has been partially supported by ``INdAM - GNSAGA Project'', codice
	CUP\_E55F22000270001, and by MCIN/AEI project PID2022-142559NB-I00. The author would like to thank B. Nelli for introducing him to the problem and for the many fruitful conversations and J.~M.~Manzano for the useful comments.
	
	\textbf{Data Availibility Statement} No data sets were generated or analysed during
	the current study.
	
	\textbf{Conflict of interest} The author declare that they have no conflict of interest.
	
	\section{Introduction}
	The Dirichlet problem for the mean curvature equation in unbounded domains has been extensively investigated over the past half-century, resulting in numerous advancements regarding the existence and uniqueness of solutions.
	
	One of the earliest significant contributions in this field was made by H.~Rosenberg and R.~Sa Earp, \cite{RoSa89}. They proved the existence of a solution to the Dirichlet problem for the minimal surface equation in unbounded domains that are contained in a strip or a sector of the plane.
	
	Another pioneering contribution in this area was made by P.~Collin and R.~Krust, \cite{CoKu91}. Their research focused on the Dirichlet problem for the prescribed mean curvature equation in Euclidean space, considering an unbounded domain $\Omega \subset \mathbb{R}^2$. The main theorem derived by Collin and Krust offers an asymptotic estimate of the difference between two solutions of this equation as these solutions approach infinity. This estimate is a fundamental tool for establishing the uniqueness of solutions over unbounded domains.
	
	To state the Collin--Krust result more precisely, it can be summarized as follows: 
	\begin{theorem}\label{Thm:Collin-Krust-original}
		Let $\Omega\subset\R^2$ be an unbounded domain and let $u,\tilde{u}\in C^2(\Omega)$ be such that $u_{\mid\partial\Omega}=\tilde{u}_{\mid\partial\Omega}$ and \[\Div\left(\frac{\nabla u}{\sqrt{1+\Norm{\nabla u}^2}}\right)=\Div\left(\frac{\nabla \tilde{u}}{\sqrt{1+\Norm{\nabla \tilde{u}}^2}}\right).\] Denote $\Lambda(r)=\Omega\cap\left\{(x,y)\in\R^2: x^2+y^2=r\right\}$ and $M(r)=\underset{\Lambda(r)}{\sup} \,| u-\tilde{u}|$.
		Hence, \[\underset{r\to\infty}{\liminf}\frac{M(r)}{\ln r}>0.\] Furthermore, if the length of $\Lambda(r)$ is uniformly bounded, then $\underset{r\to\infty}{\liminf}\frac{M(r)}{r}>0$.		
	\end{theorem}
	
	The main purpose of this work is to extend the results of \cite{RoSa89} and Theorem \ref{Thm:Collin-Krust-original} to the more general setting of a Killing Submersion. We consider an oriented and connected 3-dimensional Riemannian manifold $\E$ admitting a non-zero Killing vector field $\xi\in\X(\E)$ and a Riemannian Submersion $\pi:\E\to M$ onto a Riemannian surface $M$, connected and oriented, whose fibers are the integral curves of $\xi$, that will be referred to as Killing Submersion. This structure is a natural setting where we can study both the product 3-manifolds $M\times\R$ and the simply connected homogeneous 3-dimensional Riemannian manifold $\E$, whose space of Killing vector fields $\mathcal{K}(\E)$  has dimension at least 3, and each $\xi\in\mathcal{K}(\E)$ gives rise to a Killing Submersion structure at points where $\mu=\Norm{\xi}>0$ (see \cite[Examples 2.4, 2.5]{LerMan17}). For our purpose, in this manuscript we assume that $M$ is non-compact and the fibers have infinite length.
	
	The results originally established in \cite{RoSa89} have been further expanded upon in subsequent works. In \cite{SaTo00}, R. Sa Earp and E. Toubiana extended these results to hyperbolic space, while in \cite{SaTo08}, they considered the case of $\H^2\times\R$. The authors, along with B. Nelli, further extended these findings to the Heisenberg group in \cite{NeSaETo17}. In this work, we provide sufficient conditions that ensure the existence of a solution to the Dirichlet problem for the minimal surface equation in any Killing Submersion over certain unbounded domains of $M$. Importantly, our approach distinguishes itself from previous works by not relying on the existence of supersolutions and subsolutions, which were necessary conditions in the proofs of earlier results.
	
	Regarding Theorem \ref{Thm:Collin-Krust-original}, it has been extended to unitary Killing Submersions by C.~Leandro and H.~Rosenberg in \cite[Theorem 5.1]{LeRo09}, and improved in the specific case of minimal graphs in the three-dimensional Heisenberg group by J.~M.~Manzano and B.~Nelli in \cite[Theorem 7]{MaNe17}. In all these results, the domain exhibits uniformly bounded or linear expansion, that is, there exists a positive constant $C$ such that either $\underset{r\to\infty}{\limsup}\Length(\Lambda(r))\leq C$ or $\underset{r\to\infty}{\limsup}\tfrac{\Length(\Lambda(r))}{r}\leq C$.
	
	In Theorems \ref{thm:Collin-Krust-general} and \ref{thm:Collin-Krust-general-tau}, we provide a detailed description of the relationship between the asymptotic growth of the vertical distance between two graphs with the same prescribed mean curvature and boundary values, and the rate of expansion of the domain where they are defined, without making any assumptions about the domain. Specifically, in Theorem \ref{thm:Collin-Krust-general}, we show that the function $g(r)$, which describes the growth of the vertical distance, can be obtained by integrating the function $\tfrac{1}{\int_{\Lambda(r)}\mu^2}$. Building upon the ideas presented in \cite[Theorem 7]{MaNe17}, in Theorem \ref{thm:Collin-Krust-general-tau} we fix one of the two graphs and demonstrate that the growth function also depends on the area element of the fixed graph, which carries information about the bundle curvature $\tau$ and the prescribed mean curvature function $H$.
	Motivated by the examples, we believe that the hypotheses of Theorem \ref{thm:Collin-Krust-general-tau} are sharp. In other words, when these hypotheses are not satisfied, it is possible to find two Killing graphs with the same boundary values and prescribed mean curvature that have a bounded vertical distance.
	
	Utilizing these results, we prove the uniqueness of solutions to the Dirichlet problem for the minimal surface equation with bounded boundary values in a domain contained in a strip of $\mathbb{R}^2$ in the Heisenberg group (see Theorem \ref{them:uniquenessNil} and Corollary \ref{col:uniquenessNil}). This provides a positive answer to two open questions posed in \cite{NeSaETo17}.
	
	We also present additional results concerning graphs in Killing Submersions. Firstly, we prove a removable singularity theorem for constant mean curvature (CMC) surfaces, extending the result found in \cite[Theorem 4.1]{LeRo09} (see Theorem \ref{thm:Removable-Singularity}). Additionally, we extend some of the results from this article to higher dimensions (see Theorems \ref{thm:GD-CollinKrust} and \ref{thm:GC-RemSin}).
	
	The paper is organized as follows. In Section \ref{Preliminaries}, we describe the general properties and a local model for Killing Submersions, we define the Killing graphs and recall their mean curvature equation and some existence results. 
	In Section \ref{Existence-over-unbounded-dom}, we give sufficient conditions to prove the existence of minimal Killing graphs over unbounded domains $\Omega\subset M$.
	In Section \ref{Collin-Krust-results}, we prove the Collin-Krust type estimates.
	In Section \ref{Sec:UniquenessNil}, we restrict to the particular case of $\Nil$ and prove two uniqueness results in the strip.
	In Section \ref{Removale-sing-thm}, we generalize a removable singularity result proved by L.~Bers \cite{Be51} and R.~Finn \cite{Fi56} in Euclidean space using the technique applied by C.~Leandro and H.~Rosenberg \cite{LeRo09}. In Section \ref{Generalization-higher-dim}, we prove a removable singularity theorem and a Collin-Krust type result for Killing Submersions of arbitrary dimension.

	\section{Preliminaries and notation of Killing Submersions} \label{Preliminaries}
	In this section we recall some properties and results about Killing Submersions. For further details we refer to \cite{DLM, DMN, LerMan17}.
	
	Let $\E$ be a three-dimensional oriented Riemannian manifold, $\xi\in\mathfrak{X}(\E)$ be a non-zero Killing vector field such that the one-parameter group of isometries $G\subset\mathrm{Iso}(\E)$ of $\E$ associated to $\xi$ acts freely and properly onto $\E$. In this way, $M=\E/G$ is a well defined oriented surface that can be endowed with a unique Riemannian metric such that $\pi\colon\E\to M$ is a Riemannian submersion. The fibers of $\pi$ are the integral curves of $\xi$ and we refer to $\pi$ as a Killing Submersion. Assuming both $\E$ and $M$ are simply connected, it was proven in \cite[Theorems 2.6 and 2.9]{LerMan17} that the metric of $\E$ is uniquely determined by the choice of $M$ and $\mu,\tau\in C^\infty(M)$, with $\mu>0$. These geometric functions are the projection on $M$ through $\pi$ of the Killing length $\mu(p)=\Norm{\xi_p}$ and the bundle curvature $\tau(p)=\frac{-1}{\mu(p)}\langle\overline\nabla_u\xi,v\rangle$, where $\{u,v,\xi_p/\mu(p)\}$ is a positively oriented orthonormal basis of $T_p\E$ and $\overline\nabla$ denotes the Levi-Civita connection in $\E.$ Since both $\tau$ and $\mu$ are constant along the fibers they can be thought of as functions in $M.$ 
	If $ds^2$ is the metric of $M$, we call $\mu$-metric the conformal metric $\mu^2 ds^2$ and we will use the prefix ``$\mu$'' to indicate that the corresponding term is computed with respect to $\mu$-metric in $M$. 
	
	In this paper we will consider $M$ to be non-compact and the fibers of $\pi$ to have infinite length,  which are natural assumptions for our results. Notice that the non-compactness of $M$ guarantees the existence of a global section (\cite[Proposition 3.3]{LerMan17}).
	
	Denote by $\left\{\phi_t\right\}_{t\in\R}$ the one-parameter group of isometries associated to $\xi$. Let $\Omega\subseteq M$ be an open subset and let $F_0\colon\Omega \to\E$ be a smooth section. Hence, the Killing graph of $u\in C^\infty(\Omega)$ with respect to $F_0$ is defined as the section of $\pi$ parameterized by the map
	\[\begin{array}{rccc}
		F_u\colon&\Omega&\to&\E\\&p&\mapsto&\phi_{u(p)}(F_0(p))
	\end{array}.\]
	
	We recall some fundamental results about Killing graphs proved in \cite{DMN}. In this setting the mean curvature of the graph of $u$ with respect to the section $F_0$ is given by the following formula:
	\begin{equation}\label{eq:H}
		H_u=\frac{1}{2\mu}\,\mathrm{div}(\mu\,\pi_*(N_u))=\frac{1}{2\mu}\,\mathrm{div}\left(\frac{\mu^2\,Gu}{\sqrt{1+\mu^{2}\|Gu\|^2}}\right).
	\end{equation}	
	Here $\mathrm{div}$ denotes the divergence on $M$ and $Gu=\nabla u-\pi_*(\overline{\nabla}d)$ is the so called generalized gradient, where $d\in\mathcal{C}^\infty(\E)$ is the function such that $\phi_{d(p)}(F_0(\pi(p))=p$ for all $p\in\E.$ 
	We denote $W_u=\sqrt{1+\mu^{2}\|Gu\|^2}$ the area element of $F_u$ and we call angle function of $F_u$ the everywhere positive function
	\begin{equation*}\label{eqn:nu}
		\nu=\langle N_u,\xi\rangle=\frac{\mu}{W_u}.
	\end{equation*}
	The first result is a generalization of the technical lemma originally proved by R.~Finn~\cite{Fi56} in $\R^3$. 
	\begin{lemma}[{\cite[Lemma 2.7]{DMN}}]\label{lemma:factorization}
		For any $u,v\in\mathcal{C}^1(M)$, let $N_u$ and $N_v$ be the upward-pointing unit normal vector fields to $F_u$ and $F_v$, respectively. Then
		\[\left\langle\frac{Gu}{W_u}-\frac{Gv}{W_v},Gu-Gv\right\rangle=\frac{1}{2\mu^2}(W_u+W_v)\|N_u-N_v\|^2\geq 0.\]
		Equality holds at some point $q\in M$ if and only if $\nabla u(q)=\nabla v(q)$.
	\end{lemma}
	The following lemma provides a maximum principle for Killing graphs over compact domains of $M$.
	\begin{lemma}[{\cite[Proposition 4.2]{DMN}}]\label{MaxPrin}
		Let $\Omega$ be a relatively compact open subset of $M$ with piecewise regular boundary. Let $u,v\in\mathcal{C}^\infty(\Omega)$ be functions that extend continuously to $\overline{\Omega}\setminus C,$ where $C\subset\partial \Omega$ is the finite set of non-continuity points of $u_{\mid\partial\Omega}$ and $v_{\mid\partial\Omega}$. If
		\begin{itemize}
			\item[i)] $H(u)\geq H(v)$ in $\Omega$ and
			\item[ii)] $u\leq v$ in $\partial \Omega\setminus C$,
		\end{itemize}
		then $u\leq v$ in $\Omega$.
	\end{lemma}
	
	The following result is known as Compactness Theorem, it follows from applying gradient estimates and the Arzela-Ascoli Theorem and it is used to study the existence of the limit of a sequence of minimal graphs (see \cite[Section 2]{Del} for details).
	\begin{theorem}\label{thm:Compactness}
		Let $u_n$ be a $\mathcal{C}^0$-uniformly bounded sequence of smooth minimal Killing graphs. Then, there exists a subsequence of $u_n$ converging (in the $\mathcal{C}^k$-topology on compact subsets for all $k\in\N$) to a minimal Killing graph.
	\end{theorem}
	Finally, we recall the statement of the Jenkins--Serrin Theorem \cite[Theorem 6.4]{DMN} in this setting that guarantees the existence of upper and lower barriers.
	To do so, we consider $\Omega\subset M$ to be a Jenkins--Serrin domain, that is a relatively compact open connected domain such that $\partial\Omega$ is piecewise regular and consists of $\mu$-geodesic open arcs or simple closed $\mu$-geodesics $A_1,\ldots,A_r,B_1,\ldots,B_s$ and $\mu$-convex curves $C_1,\ldots,C_m$ with respect to the inner conormal to $\Omega$. The finite set $V\subset\partial\Omega$ of intersections of all these curves will be called the \emph{vertex set} of $\Omega$. This domain is said \emph{admissible} if neither two of the $A_i$'s nor two of the $B_i$'s meet at a convex corner.
		
	The \emph{Jenkins--Serrin problem} consists in finding a minimal graph over $\Omega$, with limit values $+\infty$ on each $A_i$ and $-\infty$ on each $B_i$, and such that it extends continuously to $\Omega\cup(\cup_{i=1}^mC_i)$ with prescribed continuous values on each $C_i$ with respect to a prescribed initial section $F_0$ defined on a neighbourhood of $\Omega$.
	To state the Jenkins--Serrin Theorem we need to recall the definition of a $\mu$-polygon: we will say that $\mathcal{P}$ is a $\mu$-polygon inscribed in $\Omega$ if $\mathcal{P}$ is the union of disjoint curves $\Gamma_1\cup\dots\cup\Gamma_k$ satisfying the following conditions:
		\begin{itemize}
			\item $\mathcal P$ is the boundary of an open and connected subset of $\Omega$;
			\item each $\Gamma_j$ is either a closed $\mu$-geodesic or a closed piecewise-regular curve with $\mu$-geodesic components whose vertices are among the vertices of $\Omega$.
		\end{itemize}
		For such an inscribed $\mu$-polygon $\mathcal P$, define
		\begin{align*}
			\alpha(\mathcal P)&=\Length_\mu((\cup A_i)\cap\mathcal{P}),&\gamma(\mathcal P)&=\Length_\mu(\mathcal{P})
			,&
			\beta(\mathcal P)=&\Length_\mu((\cup B_i)\cap\mathcal{P}).
		\end{align*}
	\begin{theorem}[{\cite[Theorem 6.4]{DMN}}]\label{thm:JS}
		Let $\Omega$ be an admissible Jenkins--Serrin domain.
		\begin{enumerate}				
			\item If the family $\{C_i\}$ is non-empty, then the Jenkins--Serrin problem in $\Omega$ has a solution if and only if 
			\begin{equation}\label{JS-condition}
				2\alpha(\mathcal P)<\gamma(\mathcal P)\quad\text{and}\quad 2\beta(\mathcal P)<\gamma(\mathcal P)
			\end{equation} 
			for all inscribed $\mu$-polygons $\mathcal P\subset\Omega$, in which case the solution is unique. 
			
			\item If the family $\{C_i\}$ is empty, then the Jenkins--Serrin problem in $\Omega$ has a solution if and only if~\eqref{JS-condition} holds true for all inscribed $\mu$-polygons $\mathcal{P}\neq\partial\Omega$ and $\alpha(\partial\Omega)=\beta(\partial\Omega)$. The solution is unique up to vertical translations.
			\end{enumerate}
	\end{theorem}
	
	In this paper we will use assume that $p$ is a fixed point of $M$ and $\Omega\subset M$ is an unbounded domain. We call $\textrm{Cut}(p)$ the cut locus of $p$,  $B_p(r)$ the geodesic ball in $M$ centered at $p$ of radius $r$, $\Omega(r)=B_p(r)\cap \Omega$ and $\Lambda(r)=\partial B_p(r)\cap\Omega$. Given a curve $\gamma\subset M$, we denote by $\Length(\gamma)$ the length of $\gamma$ computed with respect to the metric of $M$.
	
	\section{Existence of minimal Killing graphs over unbounded domains}\label{Existence-over-unbounded-dom}
	In this section we prove the existence of minimal Killing graphs over certain unbounded domains of $M$ inspired by the ideas in \cite{NeSaETo17}. 
	
	First of all, we need to define in which kind of unbounded domains we are going to work.
	\begin{definition}\label{Def:mu-wedge}
		For $p\in M$ and $\alpha\in(0,2\pi)$ let $\tilde{W}$ be a wedge of angle $\alpha$ in $T_pM$. Then, if $\exp_p\colon \tilde{W}\to M$ is a diffeomorphism, we say that $W=\exp_p(\tilde{W})$  is a \emph{$\mu$-wedge} of angle $\alpha$ and vertex $p$. Notice that $\partial W$ is the union of two half $\mu$-geodesic $\gamma_1$ and $\gamma_2$ outgoing from $p$.
	\end{definition}
	\begin{definition}\label{Def:.mu-strip}
		Let $\gamma_1,\gamma_2\in M$ be two complete non-intersecting curves, both diffeomorphic to $\R$, such that $\gamma_1\cup\gamma_2$ is the boundary of a connected domain $S\subset M.$ We will say $S$ is a ($\mu$-convex) $\mu$-strip if the $\mu$-geodesic curvature of $\gamma_1\cup\gamma_2$ with respect to the inner normal pointing $S$ is non-negative. 
	\end{definition}
	\begin{definition}\label{Def:admi}
		We say that a $\mu$-wedge $W$ is admissible if for any diverging sequences $\left\{p_1^n\right\}\subset\gamma_1$ and $\left\{p_2^n\right\}\subset\gamma_2$ the sequence $\left\{\gamma_n\right\}$ of $\mu$-geodesic connecting $p_1^n$ and $p_2^n$ does not converge. We say that a $\mu$-strip $S$ is admissible if for each end $E(S)$ of $S$, for any  diverging sequences $\left\{p_1^n\right\}\subset\gamma_1\gamma_2\cap E(S)$ and $\left\{p_2^n\right\}\subset\gamma_2\cap E(S)$ the sequence $\left\{\gamma_n\right\}$ of $\mu$-geodesic connecting $p_1^n$ and $p_2^n$ does not converge.
	\end{definition}
	We can now state the existence result as follows:
	\begin{theorem} \label{thm:existence-unbDom}
		Let $\Omega\subset M$ be an unbounded $\mu$-convex domain contained either in an admissible $\mu$-wedge $W$ of angle $\alpha<\pi$ or in an admissible $\mu$-strip $S$.
		Let $\varphi$ be a function on $\partial\Omega$ continuous except at a discrete and closed set $U\subset\partial\Omega$ of points where $\varphi$ has finite left and right limits.
		Then there exists a minimal extension of $\varphi$ over $\bar{\Omega}$.
	\end{theorem}
	\begin{proof}
		The argument is inspired by the proof of \cite[Theorem 4.3]{NeSaETo17} and relies on the Compactness Theorem and on the existence of local barriers, which is guaranteed by the existence of solutions to the Jenkins--Serrin problem. In particular, the goal is to construct a sequence $\left\{u_n\right\}$ of minimal graph that will converge to the solution. 
		
		We study separately the the cases $\Omega\subseteq W$ and $\Omega\subseteq S.$
		
		\textbf{Case 1:} $\Omega\subseteq W$.
		Let $p\in M$ be the vertex of the $\mu$-wedge $W$ containing $\Omega$ and $\gamma_1(t)$ and $\gamma_2(t)$ be the two half $\mu$-geodesics parametrized by arc length such that $\gamma_1(0)=\gamma_2(0)=p$ and $\gamma_1(\R_+)\cup\gamma_2(\R_+)=\partial W$. Let $\left\{t_n\right\}_{n\in\N}\subset\R^+$ be an increasing divergent sequence and, for all $n\in\N $ sufficiently large, let $r_n=\gamma_1(t_n)$ and $s_n=\gamma_2(t_n)$ and $\gamma_{r_n}^{s_n} \subset W$ be the $\mu$-geodesic that joins $r_n$ and $s_n$ and is closest to $p.$ We denote by $T_n$ the $\mu$-geodesic triangle with vertices $p,\,r_n$ and $s_n$ such that $\gamma_{r_n}^{s_n}\subset\partial T_n$. Notice that there can be more than one such $\mu$-geodesic satisfying this hypothesis, but we can arbitrarily choose one of them. With this choice for $\gamma_{r_n}^{s_n}$, any $\mu$-geodesic connecting $r_n$ and $s_n$ cannot lie in the interior of $T_n$. 
		
		We denote by $a_n$ (resp. $b_n$) the point of $\Omega\cap\gamma_{r_n}^{s_n}$ closest to $r_n$ (resp. $s_n$) and by $\Gamma_n$ the $\mu$-geodesic closest to $p$ joining $a_n$ and $b_n$ (notice that $\Gamma_n$ and $\gamma_{r_n}^{s_n}$ could be distinct). Finally we call $\Omega_n$ the domain bounded by $\partial\Omega\cap T_n$ and $\Gamma_n$ (see Figure \ref{Fig:Wedge-bar}).
		\begin{figure}[h!]
			\centering
			\includegraphics[scale=0.4]{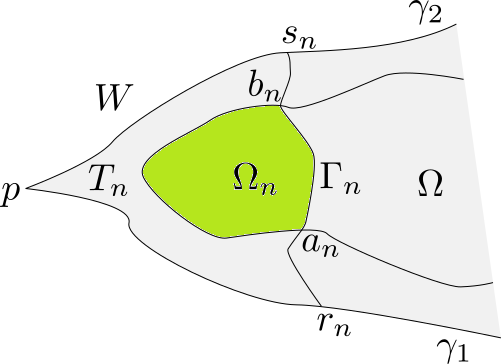}
			\caption{Sequence of domains in the $\mu$-wedge}
			\label{Fig:Wedge-bar}
		\end{figure}
		
		Since $U$ is discrete, we can assume that $\varphi$ is continuous at $a_n$ and $b_n$. 
		Notice that, by construction, \cite[Theorem 6.4]{DMN} implies that in each $\Omega_n\neq\emptyset$ we can find a minimal graph $\omega_n^\pm$ such that $\omega_n^\pm=\varphi$ in $\partial\Omega\cap T_n$ and diverges to $\pm \infty$ approaching $\Gamma_n.$
		Now we build a sequence of solution as in \cite[Theorem 4.3]{NeSaETo17}. On $\partial\Omega_n$, we consider a piecewise continuous function $\varphi_n$ such that it is continuous on $\Gamma_n$, with values between $\varphi(a_n)$ and $\varphi(b_n)$ and
		\begin{equation*}
			\varphi_n(q)=\begin{cases}
				\begin{array}{lr}
					\varphi(q)&\textrm{if }q\in\partial\Omega_n\setminus\Gamma_n;\\
					\varphi(a_n)&\textrm{if }q=a_n;\\
					\varphi(b_n)&\textrm{if }q=b_n.
				\end{array}
			\end{cases}
		\end{equation*}
		As $\Omega_n$ is bounded and $\mu$-convex and $\varphi_n$ is piecewise continuous, \cite[Theorem 6.4]{DMN} guarantees	the existence of a minimal extension $u_n$ of $\varphi_n$ on $\Omega_n$. We recall that for any discontinuity point $q\in\partial\Omega$, the boundary of the graph of each $u_n$ contains part of the fiber above $q$ with endpoints the left and the right limit of $\varphi$ at $q$. Moreover, there are no other points of the closure of the graph of $u_n$ on the vertical geodesic passing through the discontinuity points. 
		Let $n_0$ be the smaller natural number such that $\Omega_{n_0}\neq\emptyset.$ The Maximum Principle implies that $\omega_{n_0}^+\geq u_m\geq\omega_{n_0}^-$ for all $m>n_0.$ Then, using the Compactness Theorem, we can take a subsequence $\left\{u_{n_0,m}\right\}_m$ converging to a function $\tilde{u}_{n_0}$ in $\Omega_{n_0}.$ For any $n>n_0$, using $\omega_n^+$ and $\omega_n^-$ as barriers, we can solve the problem in $\Omega_n$ taking, by induction, a subsequence $\left\{u_{n,m}\right\}_m$ of $\left\{u_{n-1,m}\right\}_m$ converging to the function $\tilde{u}_n.$ By construction, $\tilde{u}_m=\tilde{u}_n$ in $\Omega_n$ for any $m>n$, that is, $\tilde{u}_m$ is the analytic extension of $\tilde{u}_n$ in $\Omega_m.$ Thus, $u=\lim_{n\to\infty}\tilde{u}_u$ will be the solution that we are looking for. 
		
		\textbf{Case 2:} $\Omega\subseteq S$.
		Let $\gamma_1(t)$ and $\gamma_2$ be the $\mu$-convex curves parametrized by arc length such that $\partial S=\gamma_1(\R)\cup\gamma_2(\R)$. Let $\left\{t_n\right\}_{n\in\N}\subset\R^+$ be an increasing divergent sequence and for any $n>0$ we call $\eta_n^l\subset S$ (resp. $\eta_n^r\subset S$) the $\mu$-geodesic that minimizes the distance between $\gamma_1(-t_n)$ and $\gamma_2(-t_n)$ (resp. $\gamma_1(t_n)$ and $\gamma_2(t_n)$) and denote by $Q_n$ the quadrilateral domain bounded by $\gamma_1([-t_n,t_n])\cup \gamma_1([-t_n,t_n])\cup\eta_n^l\cup\eta_n^r.$ Let $a_n$ (resp. $d_n$) be the point in $\eta_n^l$ closest to $\gamma_1$ (resp. $\gamma_2$) and $b_n$ (resp. $c_n$) be the point in $\eta_n^r$ closest to $\gamma_1$ (resp. $\gamma_2$). We denote by $\Gamma_n^l $ (resp. $\Gamma_n^r$) the $\mu$-geodesic closest to $\eta_n^l$ (resp. $\eta_n^l$) joining $a_n $ and $d_n$ (resp. $b_n$ and $c_n$) and call $\Omega_n$ the domain bounded by $\Gamma_n=\Gamma_n^l\cup\Gamma_n^r$ and $\partial\Omega\cap Q_n$ (see Figure \ref{Fig:Strip-bar}). 
		
		\begin{figure}[h!]
			\centering
			\includegraphics[scale=.75]{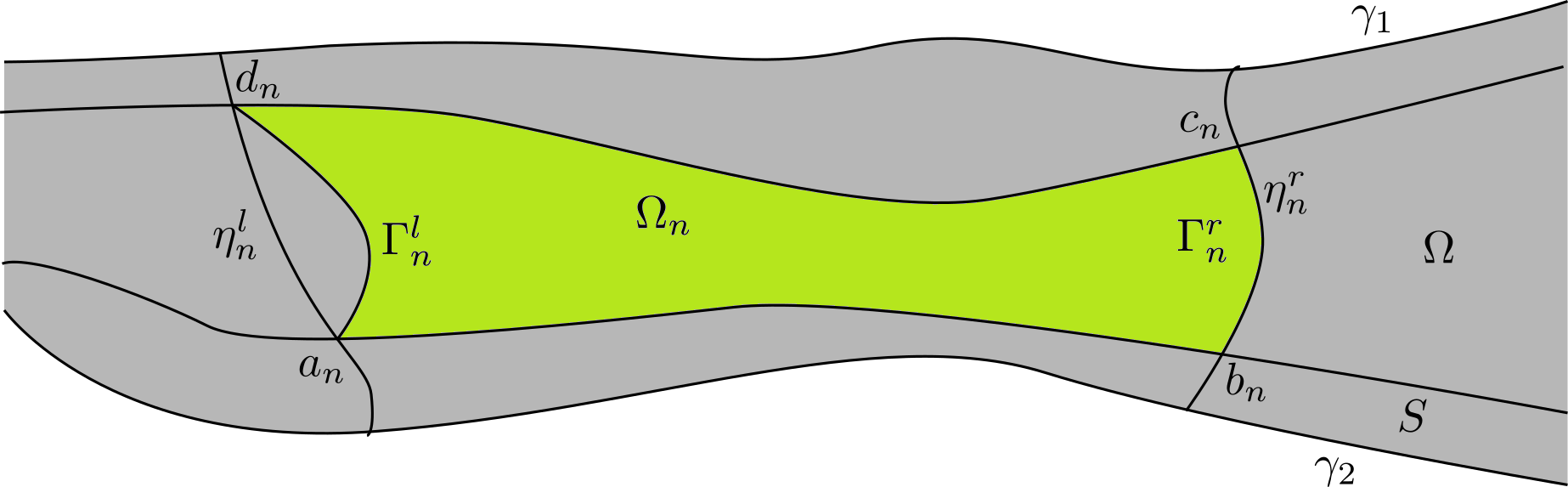}
			\caption{Sequence of domains in the $\mu$-strip}
			\label{Fig:Strip-bar}
		\end{figure}
		Since the $\mu$-metric in $S$ is admissible, there exists $n_0\in\N$ such that $\Omega_n$ satisfies the hypothesis of \cite[Theorem 6.4]{DMN} for any $n>n_0$. Therefore, there exist functions $\omega_n^\pm$ satisfying the minimal surface equation such that $\omega_n^\pm=\varphi$ in $\partial\Omega\cap Q_n$ and diverging to $\pm \infty$ as we approach $\Gamma_n.$ 
		
		Now we can proceed as in the case of the wedge: since $U$ is discrete, without loss of generality, we can suppose that $\varphi$ is continuous at $a_n,b_n,c_n,d_n$. On the boundary of	$\Omega_n$, we consider a piecewise continuous function $\varphi_n$, continuous on $\Gamma_{n}^l$, with values between $\varphi(a_n)$ and $\varphi(d_n)$, and on $\Gamma_{n}^r$, with values between $\varphi(b_n)$ and $\varphi(c_n)$,  such that
		\begin{equation*}
			\varphi_n(q)=\begin{cases}
				\begin{array}{lr}
					\varphi(q)&\textrm{if }q\in\partial\Omega\cap Q_n;\\
					\varphi(a_n)&\textrm{if }q=a_n;\\
					\varphi(b_n)&\textrm{if }q=b_n;\\
					\varphi(c_n)&\textrm{if }q=c_n;\\
					\varphi(d_n)&\textrm{if }q=d_n.
				\end{array}
			\end{cases}
		\end{equation*}
		For any $n$ sufficiently large, we denote by $u_n$ the solution to the Dirichlet problem for minimal surface equation in $\Omega_n$ such that $u_n=\varphi_n$ in $\partial\Omega_n.$ By construction, for any $n,m\in\N$ with $m>n$, the Maximum Principle implies that $\omega_n^+\geq u_m\geq\omega_n^-$ in $\Omega_n.$ From this point on we can use the same argument as in Case 1 to conclude the proof.
	\end{proof}
	\begin{remark}
		Notice that, without assuming the admissibility of $W$ (resp. $S$), the sequence of $\mu$-geodesic triangles $T_n$ (resp $\mu$-quadrilaterals $Q_n$) satisfying the hypothesis of \cite[Theorem 6.4]{DMN} does not cover the $\mu$-wedge (resp. the $\mu$-strip). 
	\end{remark}
	\begin{remark} \label{rem:ends}
		Notice that, to construct the sequence of minimal graphs that converge to our solution and the barriers that guarantee the convergence we use Theorem \ref{thm:JS}, which does not need $\Omega$ to be neither $\mu$-convex nor simply connected. Furthermore, since the $\mathcal{C}^0$ global uniform uniform bound is constructed solving the Jenkins--Serrin problem on relatively compact subset of $\Omega$, we can relax the hypothesis of $\Omega$ in Theorem \ref{thm:existence-unbDom} as follows:
		
			\emph{Let $\Omega\subset M$ be such that $\partial\Omega$ has non-negative $\mu$-geodesic curvature with respect to $\Omega$,  $\Omega$ has a countable numbers of ends and each end of $\Omega$ is contained either in an admissible $\mu$-strip or in an admissible $\mu$-wedge of angle $\alpha<\pi$.}
	\end{remark}
	\begin{remark}
		As in \cite[Remark 4.4 (C)]{NeSaETo17}, if we assume $F_0$ to be minimal, when the boundary value $\varphi$ is bounded above (respectively below) by a constant $M$, then the solution given by our proof is also bounded above (respectively below) by the same constant $M$. Furthermore, if $\varphi$ is bounded both above and below, a global barrier is given by a vertical translation of $F_0$ and the solution that we find is bounded.
		
		In general we can say nothing about the uniqueness of solutions in $\mu$-wedges and $\mu$-strips. In the forthcoming section, we will establish a Maximum Principle at infinity (see Theorems \ref{thm:Collin-Krust-general} and \ref{thm:Collin-Krust-general-tau}), which represents the initial step towards demonstrating the uniqueness of the solution of the Dirichlet problem over unbounded domains.
	\end{remark}
	
	\section{General Collin-Krust type results} \label{Collin-Krust-results}
	
	This section is devoted to prove a result useful in the study of the uniqueness of a solution of the Dirichlet problem for the prescribed mean curvature equation over unbounded domains.
	
	Let $\Omega\subset M$ be an unbounded domain and fix $p\in M$ such that $\Omega\cap\mathrm{Cut}(p)=\emptyset,$ this assumption will allow us to use the co-area formula. We denote by $B_p(r)$ the geodesic ball in $M$ centered at $p$ of radius $r$ and for any $r>r_0$ such that $\Omega(r)=B_p(r)\cap \Omega\neq\emptyset$ we call $\Lambda(r)=\partial B_p(r)\cap\Omega.$
	The first result we prove provides an estimate of the growth of the difference of any two disjoint Killing graphs $u$ and $v$ over an unbounded domain $\Omega\subset M$, having the same prescribed mean curvature and boundary values.
	The theorem introduces three key functions: $M(r)$, $L(r)$, and $g(r)$;
	\begin{itemize}
		\item $M(r)$ measures the maximum difference between $u$ and $v$ over the curve $\Lambda(r)$.
		
		\item $L(r)$ is defined as $\int_{\Lambda(r)}\mu^2$. This integral reflects the growth rate of the curve $\Lambda(r)$ with density $\mu$ and will be referred to as \emph{expansion rate function}.
		
		\item $g(r)$ is defined as $\int_{r_0}^r\frac{ds}{L(s)}$ and measures the growth rate of $M(r)$ and will be referred to as \emph{growth rate function}.
	\end{itemize}
	The theorem states that if the function $g(r)$ tends to infinity as $r$ approaches infinity, the maximum difference between $u$ and $v$ over $\Lambda(r)$ grows at a rate that is at least comparable to the growth rate of $g(r)$.

	In Theorem \ref{thm:Collin-Krust-general-tau}, we employ the idea presented in \cite[Theorem 7]{MaNe17} to improve the estimate of Theorem \ref{thm:Collin-Krust-general} when one of the two surfaces is known. Specifically, we consider one of the graphs as a fixed zero section of the Killing Submersion. By doing so, we establish that the vertical growth of any Killing graph with zero boundary values and the same prescribed mean curvature $H_0$ as that of the zero section depends on the function $L(r)=\int_{\Lambda(r)}\frac{2\mu^2}{\sqrt{1+\mu^2(a^2+b^2)}}$. Here, the smooth functions $a$ and $b$ defined in the domain $\Omega$ contain information about the bundle curvature $\tau$, as expressed in Equation \eqref{eqn:taumu}, and the mean curvature of the zero section, as expressed in Equation \eqref{eq:MC}.
	
	\begin{theorem}\label{thm:Collin-Krust-general}
		Let $\Omega\subset M$ be an unbounded domain and assume that $p\in M$ is such that $\Omega\cap \mathrm{Cut}(p)=\emptyset.$ Assume also that $u,v\in\mathcal{C}^\infty(\Omega)$ satisfy $H(u)=H(v),$  $u>v$ in $\Omega$ and $u=v$ in $\partial\Omega.$ Let $$M(r)=\underset{\Lambda(r)}{\sup}|u-v|,\qquad L(r)=\int_{\Lambda(r)}\mu^2\quad\textrm{ and }\quad g(r)=\int_{r_0}^r \tfrac{ds}{L(s)}$$ for some $r_0>0.$ Then, 
		\[\liminf_{r\to\infty}\frac{M(r)}{g(r)}>0.\]
	\end{theorem}
	
	\begin{proof}
		Denote by $\rho(r)=\int_{\Omega(r)}\Norm{\frac{\mu Gu}{W_u}-\frac{\mu Gv}{W_v}}^2$. The fact that $u-v>0$ in $\Omega$ implies that there exists $\underline{r}>0$ such that $\rho(r_0)>0$, for any $r_0>\underline{r}$. 
		Let us define $\eta(r)=\int_{\Lambda(r)}\mu\Norm{\frac{\mu Gu}{W_u}-\frac{\mu Gv}{W_v}}$. Using Lemma \ref{lemma:factorization}, the divergence theorem, and the fact that $\Norm{N_u-N_v}\geq\Norm{\frac{\mu Gu}{W_u}-\frac{\mu Gv}{W_v}}$, fixing any $r_0> \underline{r}$, we get the following estimate:
		\begin{equation}\label{first-inequality}
			\begin{array}{lcl}
				M(r)\eta(r)&\geq&\int_{\Lambda(r)}(u-v)\mu\Norm{\frac{\mu Gu}{W_u}-\frac{\mu Gv}{W_v}}=\int_{\partial \Omega(r)}(u-v)\mu\Norm{\frac{\mu Gu}{W_u}-\frac{\mu Gv}{W_v}}\\
				&\geq&  \int_{\partial \Omega(r)}(u-v)\prodesc{\frac{\mu^2 Gu}{W_u}-\frac{\mu^2 Gv}{W_v}}{\chi}\\
				&=&\int_{\Omega(r)}\Div\left((u-v)\left(\frac{\mu^2 Gu}{W_u}-\frac{\mu^2 Gv}{W_v}\right)\right)\\
				&=&\int_{\Omega(r)}\prodesc{\nabla u-\nabla v}{\frac{\mu^2 Gu}{W_u}-\frac{\mu^2 Gv}{W_v}}\\&=&\int_{\Omega(r)}\frac{W_u+W_v}{2}\Norm{N_u-N_v}^2\\
				&=&\rho(r_0)+\int_{\Omega(r)\setminus \Omega(r_0)}\frac{W_u+W_v}{2}\Norm{N_u-N_v}^2\\
				&\overset{\text{(1)}}{\geq}&\rho(r_0)+\int_{r_0}^r\left(\int_{\Lambda(s)}\Norm{\frac{\mu Gu}{W_u}-\frac{\mu Gv}{W_v}}^2\,dv_g\right)ds\\&\overset{\text{(2)}}{\geq}&\rho(r_0)+\int_{r_0}^r\frac{\eta^2(s)}{L(s)}ds,
			\end{array}
		\end{equation}
		where $\chi$ denotes a unit co-normal vector field to $\Omega(r)$ along its boundary. The inequality (1) follows by applying the co-area formula with respect to the Riemannian distance and noticing that both $W_u$ and $W_v$ are greater than or equal to 1 and that $\frac{\mu Gu}{W_u}$ is the horizontal part of $N_u$, that, $$\Norm{N_u-N_v}^2=\Norm{\frac{\mu Gu}{W_u}-\frac{\mu Gv}{W_v}}^2+\left(\frac{1}{W_u}-\frac{1}{W_v}\right)^2.$$ Inequality (2) follows from the Cauchy-Schwarz inequality
		\[\left(\int_{\Lambda(s)} f\cdot g\, dv_g\right)^2\leq \int_{\Lambda(s)} f^2\, dv_g\int_{\Lambda(s)} g^2\, dv_g\]
		with $f=\Norm{\frac{\mu Gu}{W_u}-\frac{\mu Gv}{W_v}}$ and $g=\mu$.
		Since $g(r)=\int_{r_0}^r \tfrac{ds}{L(s)},$ we get that
		\begin{equation}\label{first-inequality-consecuence}
			M(r)\eta(r)\geq\rho(r_0)+\int_{r_0}^r g'(s)\eta^2(s)ds
		\end{equation} 
		for all $r\geq r_0$.
		
		The Maximum Principle implies that the function $r\mapsto M(r)$ does not decrease. Given $r_1>r_0$, let us write $a=M(r_1)$, so $a\eta(r)\geq M(r)\eta(r)$ for all $r_0<r<r_1$. Hence, $\eta$ satisfies the integral inequality $$\eta(r)\geq\frac{\rho(r_0)}{a}+\frac{1}{a}\int_{r_0}^{r}g'(s)\eta^2(s)ds.$$ Let us define the function $\zeta\colon[r_0,R)\to\R$ as
		\begin{equation}\label{zeta-function}
			\begin{array}{lcr}
				\zeta(r)=\frac{a\rho(r_0)}{2a^2-\rho(r_0)\left[g(r)-g(r_0)\right]},
			\end{array}
		\end{equation}
		where $R>r_1$ is defined as $R=g^{-1}\left(\frac{2a^2}{\rho(r_0)}+g(r_0)\right)$ if $\left(\frac{2a^2}{\rho(r_0)}+g(r_0)\right)\in \mathrm{Im}(g)$ and $R=+\infty$ otherwise. 
		We may assume without loss of generality that $R<+\infty$, for otherwise $g$ would be bounded above, in which case the result trivially holds. Indeed, if $g$ is bounded, $\liminf_{r\to\infty}\frac{M(r)}{g(r)}=0$ if and only if $\liminf_{r\to\infty} M(r)=0$, but this contradicts the Maximum Principle.
		
		Now, observe that \[\zeta'(r)=\frac{a\rho(r_0)^2g'(r)}{(2a^2-\rho(r_0)(g(r)-g(r_0)))^2}=\tfrac{1}{a}\zeta(r)^2g'(r),\] whence \[\zeta(r)=\frac{\rho(r_0)}{2a}+\frac{1}{a}\int_{r_0}^rg'(s)\zeta(s)^2ds.\] Since $\eta(r_0)>\zeta(r_0)$ and $\eta'(r)\geq g'(r) \eta(r)^2$, while $\zeta'(r)= g'(r) \zeta(r)^2$, Chaplygin's theorem implies that $\eta>\zeta$ in $[r_0,r_1]$. 
		Since $\zeta(r)$ diverges to $+\infty $ for $r\to\R$, while $\eta(R)<+\infty$, we get that $r_1<R=g^{-1}\left(\frac{2a^2}{\rho(r_0)}+g(r_0)\right)$. Since $g$ is strictly increasing, this is equivalent to 
		\begin{equation}\label{M(r)-inequality}
			g(r_1)\leq g(R)=\frac{2M(r_1)^2}{\rho(r_0)}+g(r_0)\,\iff \, M(r_1)\geq\sqrt{\frac{\rho(r_0)}{2}\left[g(r_1)-g(r_0)\right]}
		\end{equation}
		for any $r_1>r_0$.
		
		We claim that the function $\eta$ is bounded away from zero at infinity. Note that, for $r>r_0$,
		\begin{equation}\label{etabound}
			\begin{array}{lcl}
				\eta(r)&\geq&\left|\int_{\Lambda(r)}\prodesc{\frac{\mu^2 Gu}{W_u}-\frac{\mu^2 Gv}{W_v}}{\chi}\right|\\
				&\geq&\left|\int_{\partial \Omega(r)}\prodesc{\frac{\mu^2 Gu}{W_u}-\frac{\mu^2 Gv}{W_v}}{\chi}-\int_{\partial \Omega(r)\setminus\Lambda(r)}\prodesc{\frac{\mu^2 Gu}{W_u}-\frac{\mu^2 Gv}{W_v}}{\chi}\right|.
			\end{array}
		\end{equation}
		The first integral of the right hand side of~\eqref{etabound} vanishes by Stokes' Theorem. So the result follows by proving that $\int_\Gamma\prodesc{\frac{\mu^2 Gu}{W_u}-\frac{\mu^2 Gv}{W_v}}{\chi}$ has constant sign on any arc $\Gamma$ contained in $\partial \Omega,$ as in \cite{MaNe17}. Notice that $Gu-Gv=\nabla u-\nabla v\neq 0$ along $\partial \Omega$, except at isolated points, because $u-v\geq 0$ in $\Omega$ by assumption. In particular, $Gu-Gv$ is oriented toward $\Omega$, where it is not zero. Hence, $Gu- Gv$ can be used to orient $\partial \Omega$. Then, if $\mu^2\prodesc{\frac{Gu}{W_u}-\frac{Gv}{W_v}}{Gu-Gv}$ has constant sign along $\partial \Omega$, the same holds for $\prodesc{\mu^2\frac{Gu}{W_u}-\mu^2\frac{Gv}{W_v}}{\chi}$. By Lemma \ref{lemma:factorization},
		\[\mu^2\prodesc{\frac{Gu}{W_u}-\frac{Gv}{W_v}}{Gu-Gv}=\frac{1}{2}(W_u+W_v)\Norm{N_u-N_v}^2\]
		is positive at any point where $Gu-Gv$ is not zero. Then there exists a constant $n$ such that 	$\eta(r)\geq\int_{\Gamma}\mu^2\prodesc{\frac{Gu}{W_u}-\frac{Gv}{W_v}}{\chi}\geq n>0$, which proves the claim.
		
		For any $r_2>r_0$, we deduce that
		\begin{equation}\label{b-inequality}
			\begin{array}{ccl}
				\rho(r_2)&=&\int_{\Omega(r_2)}\Norm{\frac{\mu Gu}{W_u}-\frac{\mu Gv}{W_v}}^2\geq\int_{r_0}^{r_2}\left(\int_{\Lambda(s)}\Norm{\frac{\mu Gu}{W_u}-\frac{\mu Gv}{W_v}}^2\right)ds\\
				&\geq&\int_{r_0}^{r_2}\frac{\eta^2(s)}{L(s)}ds\geq n^2\int_{r_0}^{r_2}g'(s)ds\geq n^2\left[g(r_2)-g(r_0)\right].
			\end{array}
		\end{equation}
		Observe that $g(r_0)\leq\frac{g(r_1)-g(r_0)}{2}\leq g(r_1),$ so there is $r_2\in[r_0,r_1]$ such that $g(r_2)=\frac{g(r_1)-g(r_0)}{2}.$ Applying \eqref{M(r)-inequality} to $r_2$ instead of $r_0$ we get
		\begin{equation}
			\label{last-inequality}
			\begin{array}{rcl}
				M(r_1)&\geq&\sqrt{\frac{\rho(r_2)}{2}\left[g(r_1)-g(r_2)\right]}\geq \frac{n}{\sqrt{2}}\sqrt{\left[g(r_1)-g(r_2)\right]\left[g(r_2)-g(r_0)\right]}\\&=&\frac{n}{2\sqrt{2}}\left[g(r_1)-g(r_0)\right]
			\end{array}
		\end{equation}
		for all $r_1>r_0$. Finally, this means that 
		\begin{equation*}
			\liminf_{r\to\infty}\frac{M(r)}{g(r)}\geq\liminf_{r\to\infty}\left(\tfrac{n}{2\sqrt{2}}\left(1-\tfrac{g(r_0)}{g(r)}\right)\right)>0
		\end{equation*}
		and this concludes the proof.
	\end{proof}
	
	\begin{remark}
		Up to lose some information, we can take \[g(r)=\int_{r_0}^r\tfrac{ds}{T(s)^2\Length(\Lambda(s))},\] where $T(r)=\sup_{\Lambda(r)}\mu,$ to simplify the computation. Hence, if there exists a constant $C>0$ such that $\mu_{\mid\Omega}\leq C,$ then the growth function $g(r)$ will depend only on how the growth of $\Length(\Lambda(r))$, that is $g'(r)\geq \frac{1}{C^2\Length(\Lambda(r))}.$
	\end{remark}
	
	In the next example, when $\mu$ is bounded,  we find a sharper bound on the expansion rate function of a domain $\Omega$ which guarantees a divergent Collin-Krust type estimate. 
	Such domains exist, for instance, in $\mathbb{H}^2$ (see Figure \ref{fig:H2ex}). 	
	
	\begin{example} \label{ex:divfun}
		It is not difficult to prove that $g(x)$ does not diverge whenever $f(x)=\frac{1}{g'(r)}\geq c r(\log r)^{b+1} $ for some $b,c>0$, since 
		\begin{equation*}
			\int\frac{1}{c x(\log x)^{b+1}}dx=-\frac{1}{bc\left(\log x\right)^b}+cost,
		\end{equation*}
		which is bounded above. 
		Nevertheless, we can build a sequence of monotone functions  $\left\{f_n(x)\right\}_n$ such that, for all $n\geq0,$
		$$\underset{x\to+\infty}{\lim}\frac{f_{n+1}(x)}{f_n(x)}=+\infty,$$
		and $\int\frac{1}{f_n(x)}$ diverges for $x\to+\infty$,
		with $f_0(x)=x$.
		
		Define a sequence of function as follows:
		\begin{equation*}\label{sequnce}
			\begin{cases}
				a_0(x)=x;\\
				a_i(x)=\log(a_{i-1}(x))\qquad\textrm{for }i\geq1.
			\end{cases}
		\end{equation*}
		Now we define $F_n(x)=\underset{i=0}{\overset{n}\Pi}a_i(x)$ and a sequence of \emph{translation terms}
		\begin{equation*}
			\begin{cases}
				\label{translation-terms}
				x_t^0=0,\\
				x_t^i=e^{x_t^{i-1}},\qquad\textrm{for all }i\geq1.
			\end{cases}
		\end{equation*}
		Finally, we can define 
		\begin{equation}\label{loglog-function}
			f_n(x)=F_n(x+x_t^{n}).
		\end{equation} 
		It is easy to see that 
		\begin{equation}\label{gnfun}
			g_n(x-x_t^{n+1})=\int\frac{1}{f_n(x-x_t^{n+1})}dx=
			\begin{cases}\begin{array}{ll}
					\tilde{a}_0(x)=\int \frac{1}{a_0(x)}dx&\textrm{for }n=0,\\
					\tilde{a}_n(x)=\log(\tilde{a}_{n-1})&\textrm{for }n\geq1,
				\end{array}
			\end{cases}
		\end{equation}
		which diverges. Notice that the faster the domain expands, the smaller the growth function will be.
	\end{example}
	
	Using the Poincaré's disk model, that is 
	\[\mathbb{H}^2=\left(\left\{(x,y)\in\R^2|x^2+y^2<1\right\},\frac{4}{(1-x^2-y^2)^2}(dx^2+dy^2)\right),\] we can consider the convex domain $\Omega$ bounded by bounded by $\gamma_1\cup\gamma_2$,  where $\gamma_1,\gamma_2\colon[0,+\infty)\to\mathbb{H}^2$ are such that 
	\begin{equation*}
		\begin{array}{l}
			\gamma_1(t)=\left\{\tanh\left(\tfrac{t}{2}\right) \cos\left(\frac{(t + 1) \log(t + 1)}{\sinh(t)}\right), \tanh\left(\tfrac{t}{2}\right) \sin\left(\frac{(t + 1) \log(t + 1)}{\sinh(t)}\right)\right\},\\ \gamma_2(t)=\left\{\tanh\left(\tfrac{t}{2}\right),0\right\},
		\end{array}
	\end{equation*}  
	which has expansion rate function $\Length(\Lambda(r))=(r+1)\log(r+1)$.
	\begin{figure}[h!]
		\centering\label{fig:H2ex}
		\includegraphics[width=0.34\linewidth]{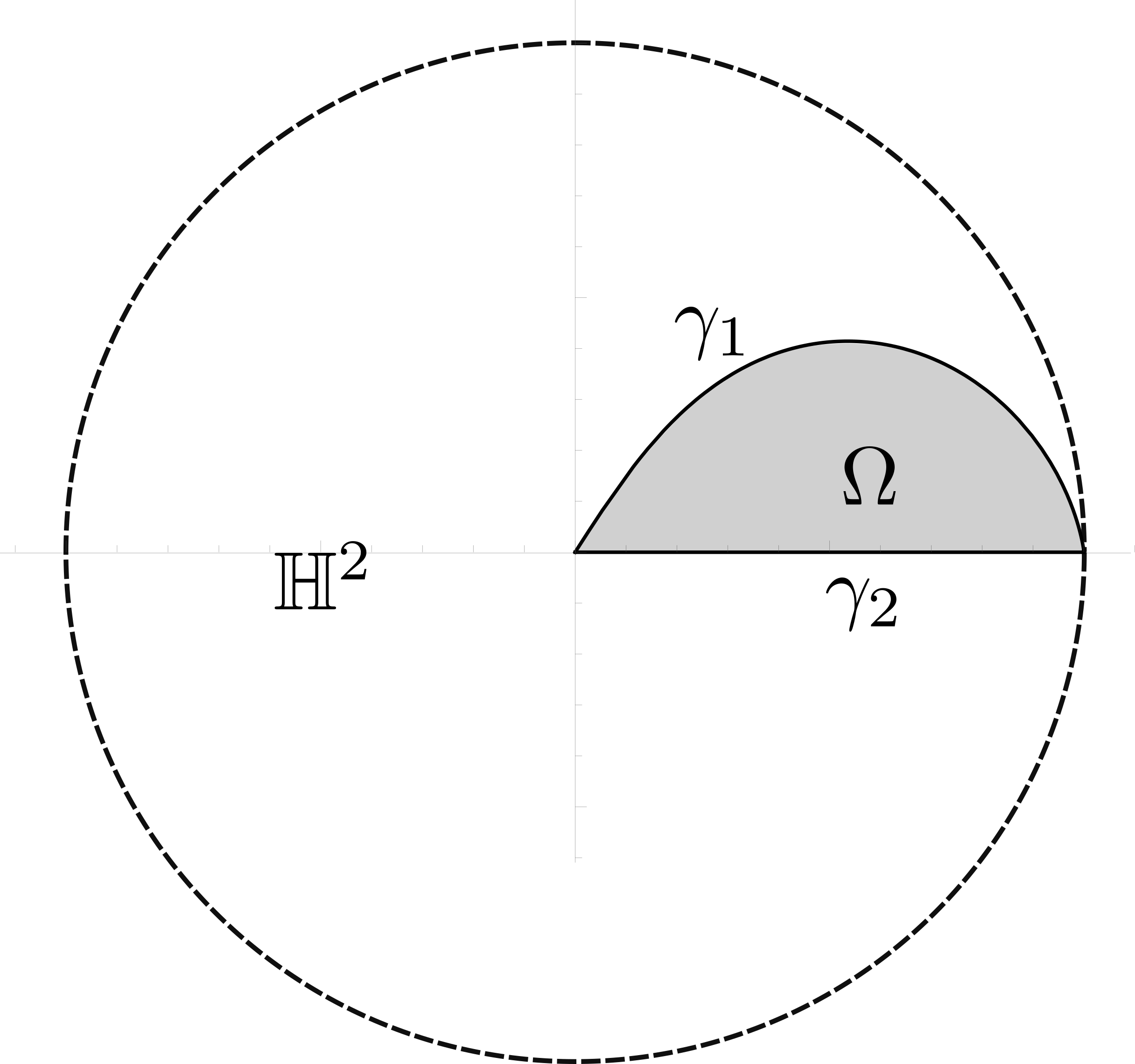}
		\caption{A domain in $\mathbb{H}^2$ whose expansion rate function is $(r+1)\log(r+1)=f_1(r)$.}
	\end{figure}
	
	In the next example we focus on the warped product $(\R^2,ds^2_{euc})\times_\mu \R$, where $\mu\in\mathcal{C}^\infty(\R^2)$ is a radial function, and we study under which assumptions on $\mu$ there exists more than one bounded solution to the same Dirichlet problem for minimal rotational Killing graphs. In particular, we see that if we choose $\mu$ in a way that guarantees a control on the expansion rate function $L(r)$, such that the growth rate function $g(r)$ does not diverge for $r\to+\infty$, then there exist two minimal graphs with the same boundary values and bounded height distance. 
	
	\begin{example}
		Let $\pi\colon\E\to M$ be a Killing Submersion such that $M$ is the euclidean plane, $\tau\equiv 0$ and $\mu(x,y)=\mu(\sqrt{x^2+y^2})$ is a smooth radial function and consider the following Dirichlet problem:
		\begin{equation*}
			\begin{cases}
				\begin{array}{ll}
					H(u)=0&\textbf{if }x^2+y^2>1;\\
					u(x,y)=0&\textbf{if }x^2+y^2=1.
				\end{array}
			\end{cases}
		\end{equation*}
		Since the domain does not contain the origin, we can consider in $\R^2\setminus\left\{(0,0)\right\}$ the polar coordinates, $$(\R^2,g_{euc}=dx^2+dy^2)\equiv\left(\R^+\times[0,2\pi),g_{pol}=dr^2+r^2d\theta\right).$$
		
		We look for non-negative solutions $u$ defined in $\Omega=\left\{(r,\theta)\in\R\times[0,2\pi]\mid r\geq 1\right\}$ such that $u(1)=0$. It is easy to see that $u(r,\theta)=u(r)$ is a radial solution of the minimal surface equation in $\Omega$ if and only if $\dot{u}(r)=\tfrac{\partial u}{\partial_r}(r)$ satisfies the following ODE:
		\begin{equation} \label{linearODE}
			\frac{\partial}{\partial r}\left( \frac{r\mu^2(r) \dot{u}(r)}{\sqrt{1+\mu^2(r) \dot{u}^2(r)}}\right)=0. 	
		\end{equation}
		Let us assume first that $\mu(r)=r$. Then the ODE becomes
		\[ \frac{\partial}{\partial r}\left( \frac{r^3 \dot{u}(r)}{\sqrt{1+r^2 \dot{u}^2(r)}}\right)=0 \]
		Now, the solution is easy to compute: for all $c\geq 1$, we get a one-parameter family of solutions $u_c(r)$ satisfying
		\[ \dot{u}_c(r)=\frac{1}{\sqrt{cr^6-r^2}}. \]
		Hence, $u_c\colon\Omega\to \R$ is given by $$u_c(r)=\int_1^r\frac{ds}{\sqrt{c s^6-s^2}}=\frac{1}{2}\arctan\left(\sqrt{c r^4-1}\right)-\frac{1}{2}\arctan\left(\sqrt{c-1}\right).$$
		Notice that $u_1(r)$ is the solution whose tangent at the boundary is vertical and can be extended by reflection to a minimal annulus. Notice also that for $c\to +\infty$, the solutions $u(r)$ converge to $u(r)\equiv0$. Finally, fixed $c\geq1$, \[0\leq\underset{r>1}{\sup}\,u(r)=\frac{\pi}{4}-\frac{1}{2}\arctan\left(\sqrt{c-1}\right)\leq\frac{\pi}{4}.\]
		
		Now, if $\mu(r)$ is an arbitrary radial function, the one-parameter family of solutions $u_c(r)$ of \eqref{linearODE} satisfy
		\begin{equation}
			\label{general-solution}
			\frac{\partial u_c}{\partial_r}(r)=\pm\frac{1}{\sqrt{cr^2\mu(r)^4-\mu(r)^2}}=\pm\frac{1}{r\mu(r)^2\sqrt{c-\frac{1}{r^2\mu(r)^2}}},
		\end{equation}
		depending on $c$.
		The comparison theorem for ODEs implies that \[\underset{r\to+\infty}{\lim}u_c(r)<+\infty\]
		whenever $\mu(r)$ grows faster than $\log(r)^{\frac{a}{2}}$ with $a>1$.
		In particular, we found a sufficient condition of $\mu$ that guarantees the existence of a family of minimal Killing graphs with the same boundary values and bounded Killing distance.
		
		Finally, using Example \ref{ex:divfun}, we show that assuming $\mu(r)>\sqrt{\log(r)}$ is not sufficient to guarantee the existence of two minimal graphs with bounded killing distance. Recall that the growth rate function is defined as \[g(r)=\int\frac{dr}{\int_{\Lambda(r)}\mu^2}=\int\frac{dr}{2\pi r\mu(r)^2}.\]
		Comparing with \eqref{gnfun}, for any fixed $n>1$, we can assume $\mu(r)=\sqrt{f_n(x)/r}$, where $\{f_n(x)\}_n$ is the sequence of function defined in \eqref{loglog-function}. Then, Theorem \ref{thm:Collin-Krust-general} implies that $\underset{r\to+\infty}{\lim}u_c(r)=+\infty$, while $\mu(r)>\sqrt{\log(r)}$ and $\underset{r\to\infty}{\lim}\frac{\mu(r)}{\sqrt{\log(r)}}=+\infty$.
	\end{example}
	
	We can apply Theorem \ref{thm:Collin-Krust-general} to study when and how the difference between two Killing graphs in $\Sol$  with the same mean curvature and the same boundary values diverges.
	Recall that $Sol_3$ can be described as a Killing submersion over $\H^2$, specifically as a warped product $\H^2\times_\mu\R$. We want to show that, unlike in $ \mathbb{H}^2 \times \R $, in $\Sol$ there are some wedges where the growth rate function diverges and it makes sense to calculate a Collin--Krust type estimate.
	
	\begin{example}
		The homogeneous manifold $\Sol$ is isometric to the warped product
		\begin{equation*}\label{Eq:SolMetric}
			\left(\left\{(x,y,z)\in\R^3\mid y> 0\right\},\frac{dx^2+dy^2}{y^2}+y^2 dz^2\right),
		\end{equation*}
		(see \cite{Nguyen}).
		In this setting,  it is easy to see that the function $u(x,y)=1-1/y$ defines a positive minimal graph in the unbounded domain of the hyperbolic plane $\left\{y>1\right\}$ that has zero boundary values and bounded height. Hence, in general, we can not expect to have a Collin--Krust type estimate in any domain. 
		
		Using the Mobius transformations, it easy to see that $\Sol$ is isometric to 
		\[\left(\mathbb{D}(1)\times \R, ds^2=\frac{4(\df x^2+\df y^2)}{(1-(x^2+y^2))^2}+\left(\frac{1-x^2-y^2}{(x-1)^2+y^2}\right)^2\df z^2\right).\]  That is, $\Sol$ is a Killing Submersion over $\mathbb{H}^2$ (described with the Poincaré Disk Model) with $\tau\equiv 0$ and $\mu(x,y)=\frac{1-x^2-y^2}{(x-1)^2+y^2}$. 
		
		Let us first define the type of domain in which we want to compute the estimate. Let $\theta_1,\theta_2\in(0,\pi)$ and for $t\in[0,1)$ define the geodesics $\gamma_1(t)=(t\cos\theta_1,t\sin\theta_1)$ and $\gamma_2(t)=(t\cos\theta_2,-t\sin\theta_2)$. We call \emph{$(\theta_1,\theta_2)$-wedge} the domain in $\mathbb{D}(1)$ bounded by $\gamma_1$, $\gamma_2$ and the asymptotic boundary $\gamma_3=(\cos\phi,\sin\phi),$ with $ \phi\in(\theta_1,2\pi-\theta_2)$.
		
		\begin{figure}[h!]
			\centering
			\includegraphics[width=0.34\linewidth]{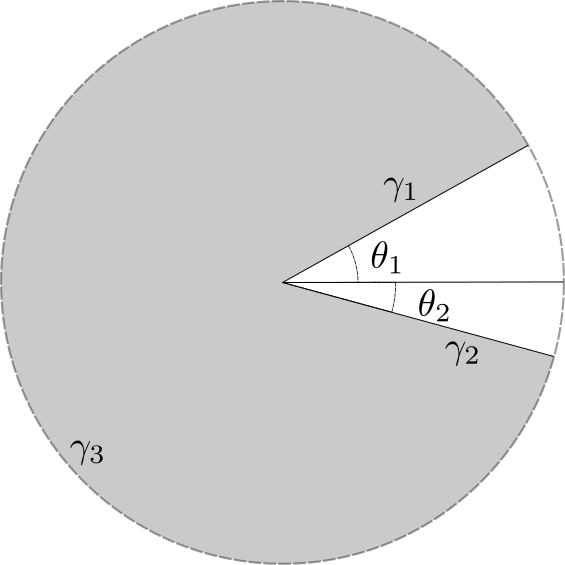}
			\caption{$(\theta_1,\theta_2)$-wedge.}
		\end{figure}
		
		Let $\Omega\subset\mathbb{H}^2$ be an unbounded domain contained within a $(\theta_1,\theta_2)$-wedge $W$ and $\Lambda(\rho)$ be the boundary of the geodesic ball of geodesic radius $\rho$ centered at the center of the Poincaré's disk contained in $\Omega$. Thus $\Length(\Lambda(\rho))\leq [2\pi-(\theta_1+\theta_2)]\sinh(\rho)$ and \[T(\rho)=\underset{\Lambda(\rho)}{\sup}\mu=\frac{1-\tanh(\rho)^2}{1+\tanh(\rho)^2-2\tanh(\rho)\cos(\theta)},\] where $\theta=\min\left\{\theta_1,\theta_2\right\}$, the minimum of the two angles $\theta_1$ and $\theta_2$ defining the wedge.
		As explained in Theorem \ref{thm:Collin-Krust-general}, $g(\rho)=\int_{\rho_0}^\rho\frac{dr}{\int_{\Lambda(r)}\mu^2},$ hence
		\begin{equation*}
			\begin{array}{rl}
				
				g'(\rho)\leq &\left(T^2(\rho)\Length(\Lambda(\rho)) \right)^{-1}=\frac{(1+\tanh(\rho)^2-2\tanh(r)\cos(\theta))^2}{[2\pi-(\theta_1+\theta_2)]\sinh(\rho)[1-\tanh(\rho)^2]^2}.
			\end{array}
		\end{equation*} 
		Integrating this inequality we have that 
		\begin{equation*}
			\begin{array}{rl}
				g(\rho)	\leq&\frac{1}{2\pi-(\theta_1+\theta_2)}\left[\tfrac{1}{2}\left(3+\cos(2\theta)\right)+\tfrac{1}{6}\left(2+\cos(2\theta)\right)\cosh(3\rho)\right.\\
				&\left.+\log\left(\tanh\left(\tfrac{\rho}{2}\right)\right)-2\cos(\theta)\sinh(\rho)-\tfrac{2}{3}\cos(\theta)\sinh(3\rho)\right]
			\end{array}
		\end{equation*}
		which diverges whenever $\theta>0$.
		
		Notice that, if $\Omega$ is an unbounded domain such that, for a $(\theta_1,\theta_2)$-wedge $W$, $\Omega\setminus W\neq\emptyset$ is compact  and $u\in C^\infty(\Omega)$ describes a minimal Killing graph with bounded boundary values, then the positive (resp. negative) part of $u-\underset{\Omega\cap W}{\max}\,u$ (resp. $u+\underset{\Omega\cap W}{\min}\,u$) is a minimal Killing graph with zero boundary values over a non-compact domain contained in $W$ and we can apply the previous estimate.	
	\end{example}
	
	To prove the last result of this section, we recall some details of the description of the local model of a Killing Submersion. In \cite{LerMan17}, it was proven that $\E$ is locally isometric to $(\mathcal{U}\times\R,ds^2)$ where $\mathcal{U}\subseteq\R^2$ and $ds^2=\lambda^2(dx^2+dy^2)+\mu^2[dz-\lambda(a dx+bdy)]^2$ for some $\lambda,a,b\in C^\infty(\mathcal{U})$, with $\lambda>0$, such that $\left(\mathcal{U},\lambda^2(dx^2+dy^2)\right)$ is a conformal parameterization of an unbounded and simply connected subset of $M$ and 
	\begin{equation}\label{eqn:taumu}
		\frac{2\tau}{\mu}=\frac{1}{\lambda^2}[(\lambda b)_x-(\lambda a)_y].
	\end{equation}
	Equation~\eqref{eqn:taumu} is the only condition on the choice of $a$ and $b$ to produce a Killing Submersion with bundle curvature $\tau$ and Killing length $\mu$. Nevertheless, the choice of $a$ and $b$ affects the mean curvature of the section $z=0$, applying \eqref{eq:H} to $u(x,y)\equiv 0$, it follows that:
	\begin{equation}\label{eq:MC}
		H_0(x,y)=\frac{1}{2\mu}\left[\partial_x\left(\frac{\mu^2\lambda a}{\sqrt{1+\mu^2(a^2+b^2)}}\right)+\partial_y\left(\frac{\mu^2\lambda b}{\sqrt{1+\mu^2(a^2+b^2)}}\right)\right].
	\end{equation}
	Furthermore, notice that if $d(x,y)$ is a $\mathcal{C}^2(\bar{\Omega})$ graph of mean curvature $H_d$, we can change $a,\, b$ for $\tilde{a},\, \tilde{b}$, where $\lambda(\tilde{a}-a)=d_x$ and $\lambda(\tilde{b}-b)=d_y$, and in this new model for $\E$ the section $z=0$ will have mean curvature $H_0=H_d$. 
	(For further information about local coordinates in Killing Submersions see \cite[Section 2.1]{LerMan17}, \cite[Section 2.2]{DLM} and \cite[Section 2.1]{DMN}).In resume, the functions $a$ and $b$ depends on the bundle curvature $\tau$ of the ambient metric and on the mean curvature of the surface $\left\{z=0\right\}.$ Furthermore, it is easy to compute that the area elemente of $\left\{z=0\right\}$ is exactly $\sqrt{1+\mu^2(a^2+b^2)}$. So, putting this information in \eqref{first-inequality}, we can prove the following theorem.
	
	\begin{theorem}\label{thm:Collin-Krust-general-tau}
		Let $\Omega\subset M$ be an unbounded domain ad assume that $p\in M$ is such that $\Omega\cap \mathrm{Cut}(p)=\emptyset.$ Assume also that $u\in\mathcal{C}^\infty(\Omega)$ satisfy $H_u=H_0,$  $u>0$ in $\Omega$ and $u=0$ in $\partial\Omega.$ Let $$M(r)=\underset{\Lambda(r)}{\sup}|u|,\quad L(r)=\int_{\Lambda(r)}\frac{2\mu^2}{\sqrt{1+\mu^2(a^2+b^2)}}\quad\textrm{ and }\quad g(r)=\int_{r_0}^r \tfrac{ds}{L(s)},$$ for some $r_0>0.$ Then, \[\liminf_{r\to\infty}\frac{M(r)}{g(r)}>0.\]
	\end{theorem}
	\begin{proof}
		We do a slight modification of the proof of Theorem \ref{thm:Collin-Krust-general}, starting from Equation \eqref{first-inequality}. Then, recalling that $W_u\geq 1$ and $W_0=\sqrt{1+\mu^2(a^2+b^2)}\geq1$ we get
		\begin{equation}\label{first-inequality-tau2}
			\begin{array}{lcl}
				M(r)\eta(r)&\geq&\int_{\Omega(r)}\frac{W_u+W_0}{2}\Norm{N_u-N_0}^2\\
				&>&\rho+\int_{\Omega(r)\setminus \Omega(r_0)}\frac{\sqrt{1+\mu^2(a^2+b^2)}}{2}\Norm{\frac{\mu Gu}{W_u}-\frac{\mu Z}{W_0}}^2\\
				&\overset{(1)}{\geq}&\rho+\int_{r_0}^r\left(\int_{\Lambda(s)}\frac{\sqrt{1+\mu^2(a^2+b^2)}}{2}\Norm{\frac{\mu Gu}{W_u}-\frac{\mu Z}{W_0}}^2\right)ds\\
				&\overset{(2)}{\geq}&\rho+\int_{r_0}^r\frac{\eta^2(s)}{L(s)}ds,
			\end{array}
		\end{equation}
		where we have used the co-area formula in (1) and the Cauchy-Schwarz inequality in (2). From this point the argument is the same as the one in the proof of Theorem \ref{thm:Collin-Krust-general}.
	\end{proof}
	
	To conclude this section, we show how Theorem \ref{thm:Collin-Krust-general-tau} can be applied to the space $\E(-1,\tau)$, providing a better estimate than the one given in \cite[Theorem 5.1]{LeRo09}. In the case of unbounded domains where $\Lambda(\rho)$ is uniformly bounded, the result of Leandro and Rosenberg \cite[Theorem 5.1]{LeRo09} states that for every choice of $\tau$ and $H$, the distance between two surfaces which have the same mean curvature grows at least as $\rho$. We will show that if $H=1/2$, the distance between two graphs having the same boundary values grows as $e^{\tfrac{\rho}{2}}$. We also show that, if we consider exterior domains, for any choice of $\tau$ and $H\in[0,1/2]$, the function $g(\rho)$ is not divergent.  
	
	\begin{example}
		Consider for $\E(-1,\tau)$  the global model given by 
		\[\left(\mathbb{D}(1)\times \R,ds^2=\lambda^2(dx^2+dy^2)+[2\tau\lambda(ydx-xdy)+dz]^2\right),\]
		where $\lambda=\frac{2}{1-(x^2+y^2)}$. In this model $a(x,y)=-2\tau y$ and $b(x,y)=2\tau x $. If we call $r=\sqrt{x^2+y^2}$, the geodesic distance of a point $(x,y)\in\mathbb{D}(1)$ from the center of the disk is given by $\rho=2\tanh^{-1}(r)$.
		
		In \cite{Pe12}, Peñafiel shows that an entire rotationally invariant graph of constant mean curvature $H\in[0,1/2]$ is parameterized as $(\tanh(\rho/2)\cos\theta,\tanh(\rho/2)\sin\theta,u(\rho))$, where $u(\rho)$ satisfies \[u'(\rho)=\frac{\left(2H\cosh(\rho)-2H\right)\sqrt{1+4\tau^2\tanh^2(\rho/2)}}{\sqrt{\sinh^2(\rho)-(2H\cosh(\rho)-2H)^2}}.\]
		Hence,
		\[\tfrac{\partial}{\partial r}u(\rho(r))=u'(\rho(r)) \frac{\partial\rho}{\partial r}=\frac{4Hr\sqrt{1+4\tau^2r^2}}{(1-r^2)\sqrt{1-4H^2r^2}}.\] 
		Now, in order to have an estimate for the vertical distance between an $H$-graph $\Sigma_H$ and the rotational entire $H$-graph $P_H$ described by Peñafiel, we have to compute 
		\begin{align*}
			\tilde{a}(x,y)=\frac{u_r(\sqrt{x^2+y^2})}{\lambda}\frac{\partial r}{\partial x}+a(x,y),\quad& \tilde{b}(x,y)=\frac{u_r(\sqrt{x^2+y^2})}{\lambda}\frac{\partial r}{\partial y}+b(x,y).
		\end{align*}
		An easy computation implies,
		\begin{align*}
			\tilde{a}(x,y)=2Hx\sqrt{\frac{1+4\tau^2(x^2+y^2)}{1-4H^2(x^2+y^2)}}-2y\tau,\\ \tilde{b}(x,y)=2Hy\sqrt{\frac{1+4\tau^2(x^2+y^2)}{1-4H^2(x^2+y^2)}}+2x\tau.
		\end{align*}
		Thus, defining \[h(\rho)=(\tilde{a}^2+\tilde{b}^2)(\rho)=\frac{4(H^2+\tau^2)\tanh^2(\tfrac{\rho}{2})}{1-4H^2\tanh^2(\tfrac{\rho}{2})}\]
		and $L(r)=\int_{\Lambda(r)}\tfrac{ds}{\sqrt{1+h(s)}}=\tfrac{2\Length(\Lambda(r))}{1+h(\rho)}$, we have $g(\rho)=\int\frac{\sqrt{1+h(\rho)}}{2\Length(\Lambda(r))}d\rho$.
		
		Denoting by $\Omega\subset\mathbb{H}^2$ the domain bounded by $\pi(\Sigma_H\cap P_H)$, our result shows that, if $\Length(\Lambda(\rho))$ is uniformly bounded,
		\begin{figure}[h!]
			\centering
			\includegraphics[width=0.34\linewidth]{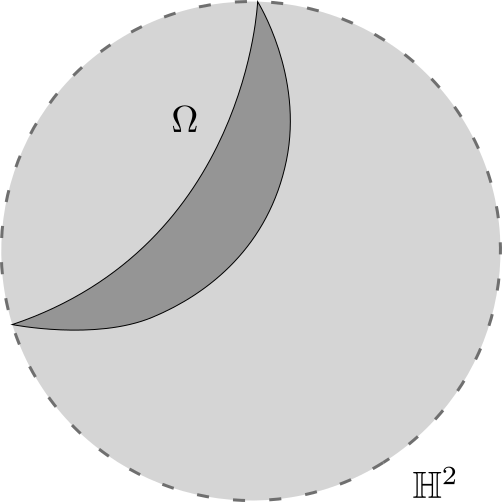}		
			\caption{Domain in $\mathbb{H}^2$ such that $\Lambda(\rho)$ is uniformly bounded.}
		\end{figure}
		then $\Sigma_H-P_H$ grows as
		\[g(r)\simeq\begin{cases}\begin{array}{ll}
				\left(\tfrac{1}{2}+\sqrt\frac{H^2+\tau^2}{1-4H^2}\right)r+ o(r)&\textrm{for }0\leq H<\tfrac{1}{2};\\ &\\
				\frac{\sqrt{1+4\tau^2}}{2}e^{\tfrac{r}{2}} +o\left(e^{\tfrac{r}{2}}\right)&\textrm{for }H=\tfrac{1}{2}.
		\end{array}\end{cases}\]
		If $\Omega\subset\mathbb{H}^2$ is an exterior domain, then $\Lambda(r)\simeq 2\pi \sinh(r)$. Hence, 
		\[g'(r)=\tfrac{1}{L(r)}\simeq \begin{cases}
			\begin{array}{ll}
				\sqrt{\frac{1-H^2-\tau^2}{4\pi^2(1-4H^2)}}e^{-r} + o\left(e^{-r}\right)&\textrm{if } H,\tau\neq0,\\
				\sqrt{\frac{1+4\tau^2}{2\pi}}e^{-\tfrac{r}{2}}+o\left(e^{-\tfrac{r}{2}}\right)&\textrm{if } H=1/2.
			\end{array}
		\end{cases}\]
		that is, $\underset{r\to\infty}{\lim}g(r)$ converges for any $0\leq H\leq 1/2$ an for any $\tau\in\R$.
		The existence of bounded graphs over exterior domains, characterized by zero boundary values and a constant mean curvature $H$ with respect to a rotational zero section of constant mean curvature $H$, has been established in various cases. Citti and Senni \cite{CiSe12} demonstrated this existence for $\tau=0$ and $H\in(0,1/2)$. Peñafiel \cite{Pe12}, focusing on surfaces invariant by rotation, proved the result for $H<1/2$ and any choice of $\tau$. In the work of Elbert, Nelli and Sa Earp \cite{ElNeSaE12}, the case of $H=1/2$ and $\tau=0$ was proven. However, the existence of a similar result for $H=1/2$ and $\tau\neq 0$ remains unknown.
		
	\end{example} 
	In all the examples we have seen, we were always able to find two different solutions whenever the growth rate function was not divergent. So it seems natural to ask the following question:
	\begin{itemize}
		\item Assume that $\Omega\subset M$ is an unbounded domain such that for any choice of $p\in M,$ with $\Omega\cap \mathrm{Cut}(p)=\emptyset$, the function $g(r)=\int_{r_0}^r\tfrac{ds}{L(s)}$ does not diverge to $+\infty$. Does a positive and bounded solution to the following Dirichlet problem 
		\[\begin{cases}
			\begin{array}{ll}
				H(u)=H(0)&\textrm{in }\Omega,\\
				u=0&\textrm{on }\partial\Omega
			\end{array}
		\end{cases}\]
		exist?
	\end{itemize}
	
	\section{A uniqueness results in a strip of the Heisenberg group}\label{Sec:UniquenessNil}
	The 3-dimensional Heisenberg group is a particular case of Killing Submersion where the base $M$ is $\R^2$ endowed with the Euclidean metric, $\mu\equiv 1$ and $\tau$ is constant. The classical model that describes $\Nil(\tau)$ is given by $\R^3$ endowed with the metric $ds^2=dx^2+dy^2+\left[\tau(y dx-x dy)+dz\right]^2$. In this model the Riemannian submersion reads as $\pi(x,y,z)=(x,y)$ and the Killing vector field is $\xi=\partial_z$. 
	
	In the Heisenberg space, Cartier constructed non-zero graphs over a wedge $\Omega\subset\R^2$ (with the vertex at the origin and any angle less the $\pi$) with zero values on $\partial \Omega$. This shows that the solution to the Dirichlet problem in $\Omega$ is not unique \cite[Corollary 3.8]{Car16}. We will prove a uniqueness result for minimal graphs with bounded boundary values over domains contained in a strip. The analogous problem was studied in $\R^3$ by Collin and Krust \cite[Theorem 1]{CoKu91} and by Elbert and Rosenberg in the product space $M\times\R,$  \cite[Theorem 1.1]{ElRo08}.
	
	The ambient isometries of this model are generated by the following maps (see \cite{FiMePe99} for more details):
	\begin{equation*}\label{Eq:IsometriesNil}
		\begin{array}{ll}
			& \varphi_1^c(x,y,z)=(x+c,y,z+c\tau y ),\\
			& \varphi_2^c(x,y,z)=(x,y+c,z-c\tau x ),\\
			& \varphi_3^c(x,y,z)=(x,y,z+c ),\\
			& \varphi_4^\theta(x,y,z)=(x\cos\theta-y\sin\theta,x\sin\theta +y\cos\theta,z ),\\
			& \varphi_5(x,y,z)=(x,-y,-z ).\\
		\end{array}
	\end{equation*}
	
	For convenience, we will introduce the following notation. Let $\Omega\subseteq\R^2$ be an open subset and let $S\subset\Nil$ be the graph of a function $u^S\in\mathcal{C}^2(U)$ where $U\subset\R^2$ is an open subset containing $\overline{\Omega}.$ We call $P(S,\Omega)$ the following Dirichlet problem:
	\begin{equation*}
		P(S,\Omega):\begin{cases}
			\begin{array}{lc}
				H_u=0&\textrm{in }\Omega,\\
				u=u^{S}&\textrm{on }\partial\Omega.
			\end{array}
		\end{cases}
	\end{equation*}
	
	Let $\Omega$ be a domain contained in a strip of $\R^2.$ Without loss of generality, applying a rotation $\varphi_4^\theta$, we can assume $\Omega\subseteq\Omega_a^b=\left\{(x,y)\in\R^2\mid a< y< b\right\}$ for some $a,b\in\R$ such that $a<b$.
	Let $I\subset\Nil$ be the entire minimal graph invariant by $\varphi_1$ given by $u^I(x,y)=\tau x y.$
	\begin{lemma}\label{lemma:TraslInv}
		The only solution to $P(I,\Omega)$ is $u^I_{\mid\Omega}$.
	\end{lemma}
	\begin{proof}	
		If $\Omega$ is relatively compact, Lemma \ref{MaxPrin} implies the result, so without loss of generality we can assume $\Omega$ to be unbounded.
		Let $c\in\R$ be such that $c>b-a$. Hence, Theorem \ref{thm:JS} implies that in the rectangle $R$ of vertices $A=(0,a)$, $B=(c,a)$, $C=(c,b)$ and $ D=(0,b)$ there exists a unique minimal surface $\Sigma^\pm$, graph of the function $\omega^\pm$, intersecting the surface $I$ above the sides $\overline{AB}$ and $\overline{CD}$ and diverging to $\pm\infty$ over $\overline{BC}$ and $\overline{DA}$.
		If $u$ is any solution of $P(I,\Omega_a^b),$ it follows that $\omega^-\leq u\leq\omega^+$ in $\partial(\Omega_a^b\cap R)$, then the Maximum Principle implies that $\omega^-\leq u\leq\omega^+$ in $\Omega_a^b\cap R$ (see Figure \ref{UpBar}).   	
		\begin{figure}[h!]
			\includegraphics[width=0.34\linewidth]{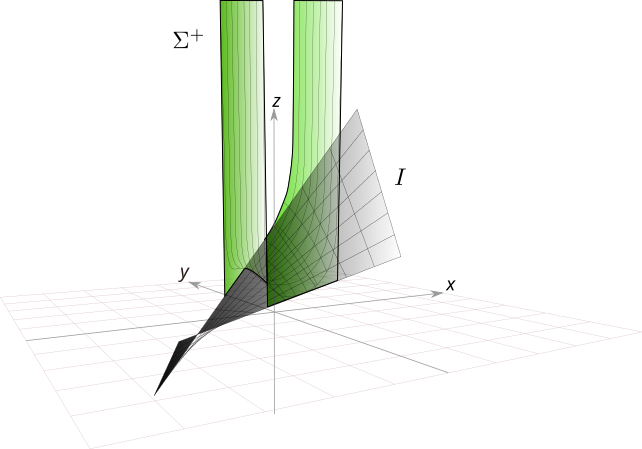}\qquad
			\includegraphics[width=0.34\linewidth]{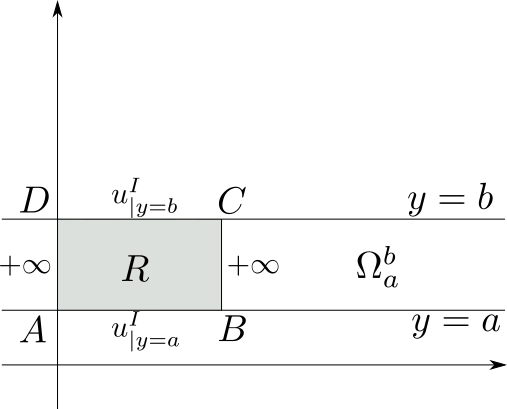}
			\caption{The upper barrier $\Sigma^+$.}
			\label{UpBar}
		\end{figure}
		Since $\varphi_1^t(I)=I$ and $\varphi_1^t$ is an isometry for all $t\in\R,$ $\varphi_1^t(\Sigma^+)$ (resp. $\varphi_1^t(\Sigma^-)$) is above (resp. below) $I$ for all $t.$
		It follows that there exists a positive constant $M$ such that any solution $\tilde{u}$ of $P(I,\Omega)$ satisfies $|\tilde{u}(p)-u^{I}(p)|< M$ for all $p\in \Omega.$ However, Theorem \ref{thm:Collin-Krust-general} implies that $| \tilde{u}-u^{I}|$ is not bounded, so we are done. 
	\end{proof}
	
	\begin{remark}
		Notice that, if $\Omega$ is a convex domain contained in a strip of $\R^2$, $f$ is a piecewise continuous function over $\partial\Omega,$ and there exists a positive constant $C$ such that $|u^{I}-f|<C$, then \cite[Theorem 4.3]{NeSaETo17} and Lemma \ref{lemma:TraslInv} imply that there exists a unique minimal graph $u$ over $\Omega$ with $u_{\mid \partial\Omega}=f$ and $|u-u^{I}|<C$.
	\end{remark}
	Using Lemma \ref{lemma:TraslInv}, it is easy to prove the following theorem, giving a positive answer to \cite[Question (a), p.17]{NeSaETo17}.
	\begin{theorem}\label{them:uniquenessNil}
		The only minimal graph in $\Nil$ over a strip of $\R^2$ with zero values on the boundary of the strip is the trivial one.
	\end{theorem}
	\begin{proof}
		After applying a rotation $\varphi_4^{\theta}$ for some $\theta$, we consider the strip $\Omega_a^b$ parallel to the $x$-axis described above. For each $n\in\N$, denote by $R_n$ the rectangle of vertices  $A_n=(-n,a)$, $B_n=(n,a)$, $C_n=(n,b)$ and $ D_n=(-n,b)$. Let $n_0$ be the integer part of $\frac{b-a}{2}+1$, then \cite[Theorem 6.4]{DMN} guarantees that for any $n>n_0$ there exists a unique solution $\omega_n^{\pm}$ to the Jenkins--Serrin problem in $R_n$ that is zero on the sides parallel to the $x$-axis and diverges to $\pm\infty$ on the sides parallel to the $y$-axis.
		It is clear that, if $u$ is a minimal solution in $\Omega_a^b$ with zero boundary values, then  $\omega_n^+>u>\omega_n^-$ for $n\in\N.$
		Furthermore, for any $n>n_0$, the Maximum Principle implies that $\omega_{n-1}^+>\omega_{n}^+>0$ (resp. $\omega_{n-1}^-<\omega_{n}^-<0$) in $R_{n-1}$. Thus, the Compactness Theorem implies that the limit $\omega^\pm=\underset{n\to\infty}{\lim}\omega_n^\pm$ exists and we call $\Sigma^\pm$ their graphs.
		
		\begin{figure}[h!]
			\centering
			\includegraphics[width=0.7\linewidth]{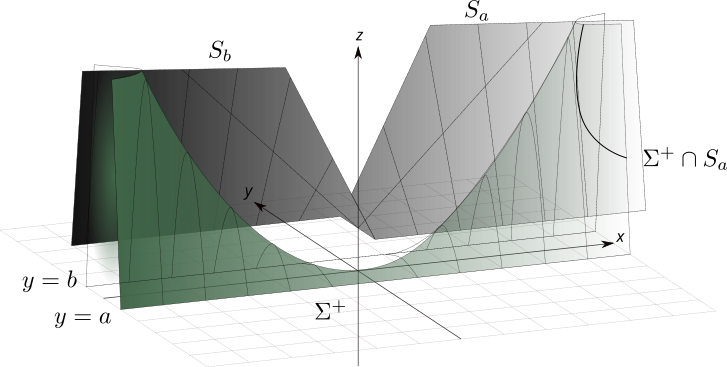}
			\caption{A contradiction on the growth of $\omega^+$.}
		\end{figure}
		
		If $\omega^+\equiv0\equiv \omega^-$, then we are done. So, suppose for instance that $\omega^+\neq0$ in $\Omega_a^b$ (the same argument can be applied with a slight modification to $\omega^-$). Theorem \ref{thm:Collin-Krust-general-tau} implies that $\omega^+$ has at least a quadratic height growth.
		We first study the asymptotic behaviour of $\omega^+$ in $\Omega_a^b\cap\left\{x\geq0\right\}$.
		Let $$M=\sup_{y\in(a,b)}\omega_{n_0}^+(0,y)$$ and, for any $c\in \R$, denote by $S=\varphi_2^{c/2}(I)$ (that is, the graph of the function $u(x,y)=\tau x(y-c)$), and $S_c=\varphi_3^M(S)$ 
		(that is, the graph of $u_c(x,y)=M+\tau x(y-c)$). Thus, by construction, $u_a\geq \omega^+$ in $\partial(\Omega_a^b\cap\left\{x\geq0\right\})$.
		Since $u_a$ has a linear height growth, it follows that $$\Omega=\left\{(x,y)\in\Omega_a^b\mid x\geq 0, u_a(x,y)<\omega^+(x,y)\right\}\neq\emptyset.$$ Hence, $\omega^+_{\mid\Omega}\neq u_a$ is a solution of $P(S_a,\Omega)$ in contradiction with Lemma \ref{lemma:TraslInv}.	
		To study the behaviour of $\omega^+$  in $\Omega_a^b\cap\left\{x\geq0\right\}$, we apply the same argument by replacing $S_a$ with $S_b$.
	\end{proof}
	
	Once we have proved that the only minimal solution with constant boundary values on a strip is constant, we can give a positive answer to \cite[Question (b), p.17]{NeSaETo17}. 
	\begin{corollary}\label{col:uniquenessNil}
		Let $\Omega$ be a domain contained in a strip of $\R^2$, such that $\partial\Omega$ has non-negative curvature with respect to $\Omega$, and $\phi\colon\partial\Omega\to\R$ be a piecewise $\mathcal{C}^2$-function and suppose that there exist $m,M\in\R$ such that $m<\phi<M$. Hence, the solution to the problem 
		\begin{equation*}
			\begin{cases}
				\begin{array}{lc}
					H_u=0&\textrm{in }\Omega,\\
					u=\phi&\textrm{on }\partial\Omega
				\end{array}
			\end{cases}
		\end{equation*} 
		is unique and satisfies $m<u<M$.
	\end{corollary}
	
		Since to prove the uniqueness we focus separately on each end of the domain, using Theorem \ref{thm:existence-unbDom} and Remark \ref{rem:ends}, we can generalize the previous result as follows:
		
		\begin{corollary}
			Let $\Omega$ be a domain contained in the union of a countable numbers of strips of $\R^2$, such that $\partial\Omega$ has non-negative geodesic curvature with respect to $\Omega$, and $\phi\colon\partial\Omega\to\R$ be a piecewise $\mathcal{C}^2$-function and suppose that there exist $m,M\in\R$ such that $m<\phi<M$. Hence, the solution to the problem 
			\begin{equation*}
				\begin{cases}
					\begin{array}{lc}
						H_u=0&\textrm{in }\Omega,\\
						u=\phi&\textrm{on }\partial\Omega
					\end{array}
				\end{cases}
			\end{equation*} 
			is unique and satisfies $m<u<M$.
		\end{corollary}

		It is possible to study the problem of uniqueness of minimal graphs in a strip of $\PSL$, viewed as a Killing Submersion over $\mathbb{H}^2$ with $\mu\equiv 1 $ and $\tau$ constant. This space, also denoted by $\E(-1,\tau)$, can be described as $\left\{\mathbb{D}\times\R,ds^2\right\}$, with $\mathbb{D}=\left\{(x,y)\in\R^2\mid \lambda(x,y)>0\right\}$ and $ds^2=\lambda^2(dx^2+dy^2) +\left[2\tau\lambda(ydx-xdy)+dz\right]^2$, where $\lambda=\frac{2}{1-(x^2+y^2)}$. In this context  $\varphi_1$ should be defined as the lifting of the hyperbolic translation, so it makes sense to consider a strip invariant under an hyperbolic translation, that is, a non-compact domain bounded by two complete curves equidistant from a fixed geodesic. The minimal graphs invariant by $\varphi_1$ in $\PSL$ have bounded height (see for example \cite{Cas} or \cite{Pe12}), so the arguments used for $\Nil$ can be easily adapted to this case, and the theorem will read as follows.
	\begin{theorem}
		Let $\Omega$ be a domain contained in the union of a countable number of strips of $\H^2$, such that each strip is invariant under an hyperbolic translation, $\partial\Omega$ has non-negative geodesic curvature with respect to $\Omega$, and $\phi\colon\partial\Omega\to\R$ be a piecewise $\mathcal{C}^2$-function and suppose that there exist $m,M\in\R$ such that $m<\phi<M$. Hence, the solution to the problem 
		\begin{equation*}
			\begin{cases}
				\begin{array}{lc}
					H_u=0&\textrm{in }\Omega,\\
					u=\phi&\textrm{on }\partial\Omega
				\end{array}
			\end{cases}
		\end{equation*} 
		is unique and satisfies $m<u<M$.
	\end{theorem}
	
		In $\E(-1,\tau)$ there is also a non-uniqueness result for strips. If we assume for example that $\Omega$ is the domain bounded by two disjoint geodesic, for any bounded function $f\in\mathcal{C}^\infty(\partial\Omega)$, there are infinitely many minimal graphs that are equal to $f$ on $\partial\Omega$. In particular, for any $g\in\mathcal{C}^\infty(\partial_\infty\H^2\cap\Omega)$, there exists a unique minimal graph $u\in\mathcal{C}^\infty(\Omega)$ such that $u(q)=f(q)$ for any $q\in\partial\Omega$ and, for any sequence $\left\{p_n\right\}\subset\H^2$ converging to $p\in\partial_\infty\H^2\cap\Omega$, we have that $\left\{u(p_n)\right\}$ converges to $g(p)$ (see, for instance, \cite[Theorem 4.9]{MRR} for $\tau=0$ and \cite[Theorem 6.3]{D} for $\tau\neq 0$).
		In general, \cite[Theorem 4.9]{MRR} and \cite[Theorem 6.3]{D} state that, if $\Omega\subset\H^2$ is an unbounded domain that contains an open segment of $\partial_\infty\H^2$, additional information on the asymptotic behaviour of the solution is required to prove a uniqueness result.
		
		In this setting, it remains open the case that lies in between: what happens if $\Omega$ is contained in an admissible strip (in the sense of Definition \ref{Def:admi}) that is not invariant by an hyperbolic translation? Is the minimal solution in $\Omega$ unique is it has bounded boundary values? The answer to this question can be deduced solving the following open problem:
		
		\emph{Let $\Omega\subset\H^2$ be a convex horodisk. Is $u\equiv c$ the unique minimal solution in $\Omega$ with constant boundary value $c$?}

	\section{Removable singularities Theorem} \label{Removale-sing-thm}
	In this section we prove a removable singularity result firstly proved by L.~Bers \cite{Be51} for minimal graphs, then by Finn~\cite{Fi56} for graphs of prescribed mean curvature in $\R^3$, by Nelli and Sa Earp\cite{NeSa} for graphs of prescribed mean curvature in the hyperbolic space and then extended in unitary Killing Submersions by C.~Leandro and H.~Rosenberg \cite[Theorem 4.1]{LeRo09}. The same technique used in \cite{LeRo09} can be applied since the function $\mu$ has an upper bound in the domains of $M$ where we are working. This extension guarantees a removable singularity result, for example, in $\Sol$ and for rotational graphs in $\R^3$. 
	\begin{theorem}\label{thm:Removable-Singularity}
		Let $u\colon\Omega\setminus\left\{p\right\} \to\R$, $\Omega\subset M$, be a function whose Killing graph has prescribed mean curvature $H\in C^{0,\alpha}(\overline{\Omega})$. Then $u$ extends smoothly to a solution at $p$.
	\end{theorem}
	\begin{proof}
		For any $R>0$, denote by $B_R(p)$ the $\mu$-geodesic ball of radius $R$ centered at $p \in M$. If $R$ is sufficiently small, \cite[Theorem 1]{DD09} guarantees the existence of a smooth function $v$ defined on $B_R(p)$ with:
		\begin{equation*}
			\begin{cases}
				\begin{array}{ll}
					\Div\left(\frac{\mu^2 Gv}{W_v}\right)=2\mu H,&\textrm{in }B_R(p);\\
					v=u,&\textrm{in }\partial B_R(p).
				\end{array}
			\end{cases}
		\end{equation*}
		
		Fix a positive constant $C$ and define the Lipschitz function
		\begin{equation*}
			\varphi=\begin{cases}
				\begin{array}{ll}
					u-v,&\textrm{if }| u-v|\
					< C;\\
					C,&\textrm{if }| u-v|\geq C.
				\end{array}
			\end{cases}
		\end{equation*}
		By definition, $\varphi$ satisfies $\nabla\varphi=\nabla u-\nabla v=Gu-Gv$ in the set $| u-v|<C$ and $\nabla\varphi= 0$ in its complement. 
		
		For $0<r<R$, let $A(r,R)=B_R(p)\setminus B_r(p)$ and denote by $m=\underset{\tiny B_R(p)}{\max}\mu$. Hence,
		\begin{equation*}
			\begin{array}{rcl}\int_{\partial A(r,R)}\varphi\mu\Prod{\frac{ \mu Gu}{W_u}-\frac{\mu Gv}{W_v},\nu}&=&\int_{\partial B_r(p)}\varphi\mu\Prod{\frac{ \mu Gu}{W_u}-\frac{\mu Gv}{W_v},\nu}+\int_{\partial B_R(p)}\varphi\mu\Prod{\frac{ \mu Gu}{W_u}-\frac{\mu Gv}{W_v},\nu}\\
				&\leq&\int_{\partial B_r(p)}Cm= Cm \Length(\partial B_r(p))	.				
			\end{array}
		\end{equation*}
		
		Since the Killing graphs of $u$ and $v$ have the same mean curvature, we have that,  when $| u-v|\geq C$, $\Div\varphi\left(\frac{\mu^2 Gu}{W_u}-\frac{\mu^2 Gv}{W_v}\right)=0$ and, when $| u-v|< C$,
		\begin{equation*}
			\begin{array}{rl}
				\Div\varphi\left(\frac{\mu^2 Gu}{W_u}-\frac{\mu^2 Gv}{W_v}\right)=&\Prod{\nabla\varphi,\frac{\mu^2 Gu}{W_u}-\frac{\mu^2 Gv}{W_v}}+\varphi\Div\left(\frac{\mu^2 Gu}{W_u}-\frac{\mu^2 Gv}{W_v}\right)\\
				=&\Prod{\nabla\varphi,\frac{\mu^2 Gu}{W_u}-\frac{\mu^2 Gv}{W_v}}\\=&\Prod{\nabla u-\nabla v,\frac{\mu^2 Gu}{W_u}-\frac{\mu^2 Gv}{W_v}}\\=&\Prod{Gu-Gv,\frac{\mu^2 Gu}{W_u}-\frac{\mu^2 Gv}{W_v}}\\=&\frac{W_u+W_v}{2}\Norm{N_u-N_v}^2_\E\leq\Norm{N_u-N_v}^2_\E,
			\end{array}
		\end{equation*}
		where the last equality follows by Lemma \ref{lemma:factorization}. By Stokes' Theorem, we have
		\begin{equation}\label{StokesCons}
			\begin{array}{c}	\int_{A(r,R)}\Div\varphi\left(\frac{\mu^2 Gu}{W_u}-\frac{\mu^2 Gv}{W_v}\right)=\int_{\partial A(r,R)}\varphi\Prod{\frac{\mu^2 Gu}{W_u}-\frac{\mu^2 Gv}{W_v},\nu}\\=\int_{\partial A(r,R)}\varphi\mu \Prod{\frac{\mu Gu}{W_u}-\frac{ \mu Gv}{W_v},\nu}\leq Cm\Length(S_r).
			\end{array}
		\end{equation}
		Thus, it follows that
		\begin{equation*}
			0\leq\int_{A(r,R)\cap\left\{| u-v|<C\right\}}\Norm{N_u-N_v}^2_\E
			\leq Cm\Length(S_r).
		\end{equation*}
		Since $m$ does not depend on $r$, as $r$ decreases to zero we get that $N_u = N_v$ on the set $|u-v|<C$. Hence, $Gu = Gv$ in the set $| u-v| <C$.
		Since $C$ was arbitrary, we have that $Gu = Gv$ in $A(0,R)$ and $u = v$ in $B_R(p) \setminus \left\{p\right\}$. Thus $u = v$ in $B_R(p)$.
	\end{proof}
	
	\section{Generalizations to higher dimension}
	\label{Generalization-higher-dim}
	In this last section we generalize some previous results to higher dimensions. Let $\pi\colon\bar{M}^{n+1}\to M^n$ be a Riemannian submersion such that the orbits of the vertical fibers are integral curves of a non-singular Killing vector field $\xi\in\mathfrak{X}(\bar{M})$. As above, $\Norm{\xi}$ is constant along the fibers, hence we can define the Killing Length $\mu=\Norm{\xi}$ as a smooth function of $M$. Let $\Omega\subset M$ be a domain. We assume that the integral curves of $\xi$ in $\bar{M}_0=\pi^{-1}(\Omega)$ are complete and non-compact.
	We first derive a formula for the mean curvature of a section of $\bar{M}$ as done in \cite[Lemma 2.1]{LeRo09} and \cite[Lemma 3.1]{LerMan17}.
	\begin{lemma}
		Let $\Sigma$ be a hypersurface of $\bar{M}$, transverse to the fibers of $\xi$. Let $N$ be a unit normal vector field to $\Sigma$. Then
		\begin{equation}\label{eq:MC}
			nH\mu=\Div(\mu\pi_* N),
		\end{equation} 
		where $\Div$ is the divergence on $M$ and $\mu=\Norm{\xi}$ is the Killing length.
	\end{lemma} 
	\begin{proof}
		For any $p\in \Sigma\subset\bar{M}$, let $\left\{E_i(p)\right\}_{i=1,\dots,n+1}$ be an oriented orthonormal basis of $T_p\bar{M}$ such that $E_{n+1}=\frac{\xi}{\mu}$. In particular, $\left\{\pi_*E_i(p)\right\}$ is an oriented orthonormal basis of $T_{\pi(p)}M$. If we denote by $\bar{N}$ the extension of $N$ to $\bar{M}$ (constant along the fibers of $\xi$), then at $p$ we get 
		\[nH=\Div_{\bar{M}}(\bar{N})=\sum_{i=1}^n\prodesc{\bar{\nabla}_{E_i}\bar{N}}{E_i}+\prodesc{\bar{\nabla}_{E_{n+1}}\bar{N}}{E_{n+1}}=\Div_M\left(\pi_*N\right)+\prodesc{\bar{\nabla}_{E_{n+1}}\bar{N}}{E_{n+1}}.\]
		Now, since $\xi$ is Killing and $\bar{N}$ is constant along the fibers of $\xi$, it follows that
		\[\prodesc{\left[\bar{N},\xi\right]}{\xi}=\prodesc{\bar{\nabla}_{\bar{N}}\xi}{\xi}-\prodesc{\bar{\nabla}_{\xi}\bar{N}}{\xi}=-\prodesc{\bar{\nabla}_{\xi}\xi}{\bar{N}}-\prodesc{\bar{\nabla}_{\xi}\bar{N}}{\xi}=-\xi\left(\prodesc{\bar{N}}{\xi}\right)=0.\]
		Finally, since $E_{n+1}=\frac{\xi}{\mu}$, we apply the Koszul formula to get
		\[\prodesc{\nabla_{E_{n+1}}\bar{N}}{E_{n+1}}=\prodesc{\left[E_{n+1},\bar{N}\right]}{E_{n+1}}=\tfrac{1}{\mu^2}\prodesc{\left[\xi,\bar{N}\right]}{\xi}-\tfrac{1}{\mu}\prodesc{\bar{N}(\mu)\xi}{\xi}=\tfrac{1}{\mu}\bar{N}(\mu)=\tfrac{1}{\mu}\prodesc{\bar{N}}{\bar{\nabla}\mu}=\tfrac{1}{\mu}\prodesc{\pi_*N}{\nabla\mu}\]
		since $\xi(\mu)=0$.
		
		In particular,
		\[nH=\Div(\pi_* N)+\tfrac{1}{\mu}\prodesc{\pi_*N}{\nabla\mu}=\tfrac{1}{\mu}\Div\left(\mu\pi_*N\right).\]
	\end{proof}
	Denote $\phi_t$ the 1-parameter group of isometries associated with $\xi$ and consider a smooth embedding $i\colon\Omega\to\bar{M}$, which is a section of the fibration, and assume $\Sigma_0 = i(\Omega)$ is transverse to $\xi$.
	Thus, we can define the Killing graph $\Sigma_u$ of a function $u\in C^2(\Omega)$ as the section $\Sigma_u=\left\{\phi_{u(x)}(i(x)),x\in\Omega\right\}$.
	
	Now, let us consider $\bar{u} \in C^\infty(\bar{M})$ the extension of $u$
	by making it constant along fibers, and $d \in C^\infty(\bar{M})$ the function that measures the signed vertical Killing distance from a point to the point in $\Sigma_0$ lying on the same fiber. In other
	words, $d$ is determined by the identity $\phi_d(p)(\Sigma_0(\pi(p))) = p$, for all $p\in\bar{M}$. Then the section $\Sigma_u$ is a level hypersurface of the function $u-d\in C^\infty(\bar{M})$ so we can compute	$N = \bar{\nabla}{(u-d)}/\Norm{\bar{\nabla}{(u-d)}}$. If we set $Z = \pi_*(\bar{d})$, we get
	\[\pi_* N=\frac{\mu Gu}{W_u},\qquad Gu=\nabla u-Z,\]
	where the gradient and the norm are computed in $M$ and $W_u^2=1+\mu^2\Norm{Gu}^2$ and
	\[H_u=\frac{1}{n \mu}\Div\left(\frac{\mu^2 Gu}{W_u}\right)\]
	
	Hence, as in \cite{DMN}, the following lemma holds:
	\begin{lemma}\label{lemma:factorization-higher-dim}
		For any $u,v\in\mathcal{C}^1(M)$, let us consider $N_u$ and $N_v$ the upward-pointing unit normal vector fields to $F_u$ and $F_v$, respectively. Then
		\[\left\langle Gu-Gv,\frac{\mu^2 Gu}{W_u}-\frac{\mu^2 Gv}{W_v}\right\rangle=\frac{1}{2}(W_u+W_v)\|N_u-N_v\|_{\bar{M}}^2\geq 0.\]
		Equality holds at some point $p\in M$ if and only if $\nabla u(p)=\nabla v(p)$.
	\end{lemma}
	With the same technique as above, Lemma \ref{lemma:factorization-higher-dim} allows us to prove the following generalizations of Theorems \ref{thm:Collin-Krust-general} and \ref{thm:Removable-Singularity}.
	\begin{theorem}\label{thm:GD-CollinKrust}
		Let $p\in M$ and let $\Omega\subset M$ be an unbounded domain such that  $\Omega$ does not intersect the cut locus of $p$. Denote by $B(p,r)=\left\{x\in M\mid \textrm{dist}(x,p)< r\right\}$ and, for $r>r_0$, $\Omega(r)=\Omega\cap B(p,r)\neq\emptyset$ and $\Lambda(r)=\partial B(p,r)\cap\Omega$. 
		Let $M(r)=\underset{\Lambda(r)}{\sup}|u-v|,$ $L(r)=\int_{\Lambda(r)}\mu^2$ and $g(r)=\int_{r_0}^r \tfrac{ds}{L(s)}$ for some $r_0>0.$ Then,  \[\liminf_{r\to\infty}\frac{M(r)}{g(r)}>0.\]
		If $\mu$ is bounded above, it is sufficient to assume that $g(r)=\int_{r_0}^r\frac{ds}{\Vol(\Lambda(s))}$.		
	\end{theorem}
	
	Now, if $\Sigma$ is an hypersurface of $\bar{M}$ that is invariant under vertical translation, then $\Gamma=\pi(\Sigma)$ is a well defined hypersurface of $M$ and, using $\eqref{eq:MC}$, we can compute the mean curvature of $\Sigma$ in terms of the mean curvature of $\Gamma$ in $M$ and the function $\mu$. In particular, 
	\[nH_\Sigma=\Div(\pi_* N)+\prodesc{\frac{\nabla\mu}{\mu}}{\pi_*N}=(n-1)\left(H_\Gamma+\frac{1}{n-1}\prodesc{\frac{\nabla\mu}{\mu}}{\pi_*N}\right),\]
	where $H_\Sigma$ is the mean curvature of $\Sigma$ in $\bar{M}$ and $H_\Gamma$ is the mean curvature of $\Gamma$ in $M$. The quantity $\left(H_\Gamma+\frac{1}{n-1}\prodesc{\frac{\nabla\mu}{\mu}}{\pi_*N}\right)$ is exactly the mean curvature of $\Gamma $ in $M$  with density $\tfrac{1}{\mu}$ (see \cite{C} for more details about properties of manifolds with density).
	
	Finally, we recall an existing theorem proved by Dajczer and De Lira, \cite[Theorem 1]{DD09}, for Killing graphs of prescribed mean curvature in Killing submersion with non-necessarily integrable horizontal distribution.
	\begin{theorem}\label{thm:existence-hig-dim}
		Let $\Omega\subset M$ be a relatively compact domain with $\mathcal{C}^{2,\alpha}$ boundary $\Gamma$. Suppose that $H_{cyl}>0$ (where $H_{cyl}$ is such that $H_\Gamma=H_{cyl}\circ\pi$ and describe the mean curvature of $\pi^{-1}(\Gamma)$) and that \[\textrm{Ric}_{\bar{M}}\geq-n\inf_\Gamma H_{cyl}^2.\]
		Let $H\in\mathcal{C}^\alpha(\Omega)$ and $\phi\in\mathcal{C}^{2,\alpha}(\Gamma)$ be given. If $\sup_\Omega|H|\leq\inf_\Gamma H_{cyl}$, then there exists a unique function $u\in\mathcal{C}^{2,\alpha}(\bar{\Omega})$ satisfying $u_{|\Gamma})\phi$ whose Killing graph has mean curvature $H$.
	\end{theorem}
	Theorem \ref{thm:existence-hig-dim}, together with Lemma \ref{lemma:factorization-higher-dim}, enables us to apply the same argument used in the proof of Theorem \ref{thm:GC-RemSin} to extend the result to higher dimensions.
	\begin{theorem}\label{thm:GC-RemSin}
		Let $\Omega\subset M$ be an open domain and $u\colon\Omega\setminus\left\{p\right\} \to\R$ be a function whose Killing graph has prescribed mean curvature $H$. Then $u$ extends smoothly to a solution at $p$.
	\end{theorem}

\end{document}